\documentclass[preprint,12pt]{elsarticle}

\usepackage{amssymb}
\usepackage{amsmath}
\usepackage{amsfonts}
\usepackage{amsthm}
\usepackage{color}
  
\journal{}
\begin{document}

\def\d{{\partial}}
\def\CC{{\mathbb C}}
\def\C{{\mathbb C}}
\def\A{{\mathbb A}}
\def\AA{{\mathbb A}}
\def\cC{{\cal C}}
\def\cA{{\cal A}}
\def\cB{{\cal B}}
\def\cH{{\cal H}}
\def\ZZ{{\mathbb Z}}

\def\DD{{\mathbb D}}
\def\D{{\mathbb D}}
\def\cD{{\cal D}}
\def\NN{{\mathbb N}}
\def\S{{\mathbb S}}

\def\QQ{{\mathbb Q}}

\def\Arg{{\mbox{Arg} \, }}

\def\RR{{\mathbb R}}
\def\R{{\RR}}

\def\TT{{\mathbb T}}
\def\T0{{\mathbb T}_{x_0}}
\def\T{{\mathbb T}}
\def\cW{{\cal W}}
\def\cF{{\cal F}}
\def\cR{{\mathcal R}}
\def\e{{\varepsilon}}

\def\cS{{\cal S}}
\def\cH{{\cal H}}
\def\trait (#1) (#2) (#3){\vrule width #1pt height #2pt depth #3pt}
\def\fin{\hfill\trait (0.1) (5) (0) \trait (5) (0.1) (0) \kern-5pt 
\trait (5) (5) (-4.9) \trait (0.1) (5) (0)}
\def\tr{\mbox{tr}}
\def\supess{\mathop{\mbox{ess sup}\,}}

\newtheorem{theorem}{Theorem}[section]
\newtheorem{corollary}{Corollary}[section]
\newtheorem{lemma}{Lemma}[section] 
\newtheorem{proposition}{Proposition}[section]
\newtheorem{conjecture}{Conjecture}[section]
\newtheorem{remark}{Remark}[section]
\newtheorem{example}{Example}[section]

\newcommand {\be}{\begin{equation}}
\newcommand {\ee}{\end{equation}}  
\newcommand {\eps} {\varepsilon}
\newcommand {\alp} {\alpha}
\newcommand {\la} {\lambda}
\newcommand {\sig} {\sigma}
\newcommand {\deri}[2] {\partial_{#2} #1} 
\newcommand {\derd}[3] {\frac {\partial^{2} #1}{\partial #2 \partial 
#3}}
\newcommand {\ders}[2] {\frac {\partial^{2} #1}{\partial {#2} ^2 }}
\newcommand{\modif}[1]{{\color{red} #1}}

\begin{frontmatter}



\title{Uniqueness results for 
{{inverse}} Robin problems with bounded coefficient} 


\author{Laurent Baratchart}

\address{Projet APICS, INRIA, 2004 route des Lucioles, BP 93\\
06902 Sophia Antipolis Cedex, France}

\ead{laurent.baratchart@inria.fr}

\author{Laurent Bourgeois}

\address{Laboratoire POEMS, ENSTA ParisTech, 828, Boulevard des Mar\'echaux\\
   91762 Palaiseau Cedex, France}

\ead{laurent.bourgeois@ensta.fr}

\author{Juliette Leblond}

\address{Projet APICS, INRIA, 2004 route des Lucioles, BP 93\\
06902 Sophia  Antipolis Cedex, France}

\ead{juliette.leblond@inria.fr}

\begin{abstract}
In this paper we address the uniqueness issue in the classical Robin inverse problem on a Lipschitz domain $\Omega\subset\RR^n$, with $L^\infty$ Robin coefficient, $L^2$ 
Neumann data and  conductivity of class $W^{1,r}(\Omega)$, $r>n$.
{{We show  that}} uniqueness of the Robin coefficient on a subpart of the boundary{{,}} 
given Cauchy data on the {{complementary part, does hold}} in dimension 
{{$n=2$}} {{ but needs not hold in higher dimension. We also raise on open issue on harmonic gradients which is of interest in this context}}.
\end{abstract}


\end{frontmatter}



\section{Introduction}

This study deals with uniqueness issues { for} the classical Robin inverse 
boundary value { problem.  Mathematically speaking, the inverse Robin 
problem for an elliptic partial differential equation on a domain consists
in finding the ratio between the normal derivative
and the trace of the solution (the so-called Robin coefficient) on
a subset of the boundary, granted the Cauchy data ({\it i.e.} the normal 
derivative and the trace of the solution)  on the complementary subset.
In this paper, we deal primarily with} $L^\infty$ Robin coefficients and $L^2$ 
Neumann data, for isotropic
{{conductivity equations of the type $\mbox{div} \, \left(\sigma \, \mbox{grad} \, u\right) = 0$ on Lipschitz domains $\Omega\subset\RR^n$,}} 
with {{Sobolev-smooth real-valued {strictly elliptic}
conductivity $\sigma$ of class 
$W^{1,r}(\Omega)$, $r>n$}}. {An anisotropic analog 
to our uniqueness result is discussed  in a separate section.}

The Robin inverse problem arises for example when considering 
non-destructi\-ve testing of corrosion in an electrostatic conductor.
In this case, data consist of surface measurements of both the 
current and the voltage on some (accessible) part of the boundary of 
the conductor, while the complementary (inaccessible) part of the boundary 
is subject to corrosion. 
Non-destructive testing consists in quantifying corrosion from the data. 
Robin boundary condition can be regarded as a simple model for corrosion 
\cite{inglese}. Indeed, as was proved in \cite{buttazzo_kohn}, 
such boundary conditions arise
when considering a thin oscillating coating surrounding a homogeneous background medium such that the thickness of the layer and the wavelength of the oscillations tend simultaneously to $0$.
A mathematical framework for corrosion detection can then be
described as follows. We consider a 
conductivity equation in an open domain $\Omega$, {{as a generalization of Laplace equation to non-homogeneous media,}}
the boundary of which is divided 
into two parts. The first part $\Gamma$ is characterized by a homogeneous Robin condition with functional coefficient $\la$. A non vanishing flux is imposed on the second part $\Gamma_0$ of the boundary. This provides us
with a well-posed forward problem, that is,  there uniquely exists a
solution in $\Omega$ meeting the prescribed boundary conditions.
The inverse problem consists in
recovering the unknown Robin coefficient $\la$ on $\Gamma$
from measurements of the trace of the solution on $\Gamma_0$.
Further motivation to solve the Robin problem are indicated in
\cite{lanzani-shen} and its bibliography.

A basic question is uniqueness: is the coefficient $\la$ on $\Gamma$ uniquely defined by the available Cauchy data on $\Gamma_0$ as soon as the latter has 
positive measure? In other words, can we find two different
Robin coefficients that produce the same measurements? 
The answer naturally depends on the smoothness assumed for
$\la$. 

On smooth domains,
for the Laplace operator  at least, uniqueness of the inverse
Robin problem for  (piecewise) continuous 
$\la$ has been known for decades to hold in all dimensions.
The proof is for example given { in \cite{inglese}, and  in \cite{colton_kirsch} for the Helmholtz 
equation}.
It relies on a strong unique continuation property { (Holmgren's theorem)}, 
{\it i.e.} on the fact that a harmonic function in $\Omega$, the trace and 
normal derivative of which both vanish on { a non-empty open} subset of 
the boundary $\partial\Omega$, vanishes identically. 

This argument no longer works for functions $\la$ that are merely
bounded.
In this case we meet the following weaker unique continuation problem: 
{\em does a harmonic function, the trace and  normal derivative of which both 
vanish on a subset of $\partial\Omega$ with positive measure, 
vanish identically?} 
A famous counterexample { in} \cite{bourgain_wolff} 
shows that such a unique continuation result is false in dimension 
{ {$3$ and higher}}. 
In dimension 2, a proof that such a unique continuation property
holds   for { {the Laplace equation can be found}} 
in \cite{alexandrov} when the { solution} is assumed to be $C^1$ up to the 
boundary and $\Omega$ is the unit disk.

In this work, we prove more generally
that { this  unique continuation result still holds for a $W^{3/2,2}$ solution} 
to a 
conductivity equation with $W^{1,r}$-conductivity {{$\sigma$, { {$r>2$}},}}  in a
bounded simply connected Lipschitz domain {{$\Omega \subset \RR^2$}}.
{ This enables us to conclude to uniqueness 
in the inverse Robin problem. 
}Our proof { relies on two devices}:

-- A  factorization result for the complex derivative
of a solution to an isotropic conductivity equation, where one factor is
holomorphic and the other is smoothly  invertible. This factorization implicitly appears in
\cite{BLRR}, but we shall have to work out its regularity 
on a Lipschitz domain. The holomorphic { factor} in fact 
belongs to a Hardy--Smirnov class, hence is uniquely defined 
{{by its boundary values 
on a boundary subset of positive measure.}}

--  A Rolle-type theorem for
$W^{1,2}$ Sobolev functions on the real line.

Our { uniqueness result for the Robin inverse problem} generalizes that of \cite{chaabane_ferchichi_kunisch}
established in smoother cases and 
under the restriction that the imposed flux is non negative. The
proof therein is based on 
positivity and monotonicity arguments established in 
\cite{chaabane_ferchichi_kunisch2}, and does not use complex analysis. 
We also turn the counterexample of \cite{bourgain_wolff} into a 
counterexample to uniqueness in the Robin problem in dimension { 3,
and raise an intriguing issue on harmonic gradients vanishing on a 
boundary subset of positive measure which governs uniqueness in
higher dimension under mild smoothness assumptions on
the sets where the Cauchy data and
the Robin coefficient are defined.}

{ The  { paper} is organized as follows. 
In section \ref{sec:Sob}, we set some notation and we recall several results
from the theory of Sobolev spaces. 
In
Section \ref{sec:cond}, we introduce the isotropic conductivity PDE and associated Robin problems. 
In Section \ref{sec:ur00},
we state our uniqueness   results for such equations 
on Lipschitz domains in dimension 2.
We also give a counterexample in higher dimension.
Section \ref{secH00} 
is a review of
holomorphic Hardy spaces on the disk 
{{and their generalization into}} 
Smirnov spaces 
on Lipschitz domains, in connection
with the Dirichlet problem for harmonic functions. 
Proofs {{of}} the results in Section \ref{sec:ur00}  are 
provided in Section \ref{sec:proof1},
along with the {{necessary}} factorization and regularity 
{{properties of  solutions to the {{2D}} Neumann problem which are of
interest in their own right. Surprisingly perhaps, these seem not to 
have appeared before in the literature}}.}
{In section \ref{anisotropic}, we indicate how our uniqueness
result for the isotropic Robin problem implies a corresponding result in the 
anisotropic case. For this, we rely on the method of isothermal coordinates
initiated in \cite{Sylvester} and pursued in \cite{APL, SunUhl}, allowing 
us to transform an anisotropic equation in the plane into an isotropic 
one.
Section \ref{conclusion} contains  concluding remarks.}

\section{Notation and preliminaries on Sobolev spaces}
\label{sec:Sob}

Let $\RR$ and $\CC$ denote the real and complex numbers.
With superscript ``$t$'' to mean ``transpose'',
we write $x=(x_1,\cdots,x_n)^t$ to indicate the coordinates of $x\in\RR^n$,
and we identify  $\CC$ with $\RR^2$ on putting
$z=x_1+ix_2$.

For $1\leq p\leq\infty$, $k>0$ an integer and $E\subset\RR^n$ a 
Lebesgue measurable set, we let $L^p(E)$ be the space of  
$\RR^k$-valued
measurable functions on $E$ such that
\begin{align}
\label{defnLp}
 \|f\|^p_{L^p(E)}& = 
\int_E |f|^p\,dm_n<\infty\qquad\text{if}\ 
p<\infty,\\
\nonumber
 \|f\|_{L^\infty(E)}& = 
\text{ess sup}_E\  |f|<+\infty,
\end{align}
where $m_n$ stands for Lebesgue measure.
In \eqref{defnLp} above, $|f|$ designates the
Euclidean norm of $f$ and the  notation is irrespective of $k$, 
which should cause no confusion.

{ In Section \ref{ssec:sob} we recall some properties of Sobolev spaces. 
We turn in Section \ref{ssec:ntmf} to classical definitions of non tangential
convergence and maximal functions, while Section \ref{sec:pc} is specifically devoted to the planar case.
}
\subsection{Sobolev spaces}
\label{ssec:sob}
For $\Omega\subset\RR^n$ an open set,
we let $W^{1,p}(\Omega)$ be the familiar Sobolev space of complex-valued
functions in $L^p(\Omega)$ whose first order
derivatives  again lie in $L^p(\Omega)$. A complete norm  on
$W^{1,p}(\Omega)$ is given by
\begin{align}
\label{normS}
 \|f\|^p_{W^{1,p}(\Omega)}&= 
\|f\|^p_{L^p(\Omega)}+\|\nabla f\|^p_{L^p(\Omega)}
\qquad\text{if}\ 
p<\infty,\\
\nonumber
\|f\|_{W^{1,\infty}(\Omega)}&=\max\bigl(\|f\|_{L^\infty(\Omega)},
\|\nabla f\|_{L^\infty(\Omega)}\bigr),
\end{align}
where $\nabla f$ is the gradient of $f$ defined as
$\nabla f=(\partial_{x_1}f,\cdots,\partial_{x_n} f)^t$ , 
with $\partial_{x_j}$ to indicate the derivative with respect 
to $x_j$. 

When $n=1$, we simply write $f'$ instead of $\partial_{x_1} f$.
Throughout, differentiation is given in the 
distributional sense: $\int_\Omega \partial_{x_j}f\varphi dm_n=-\int_\Omega f\partial_{x_j}\varphi dm_n$ whenever $\varphi\in\mathcal{D}(\Omega)$, the space of complex-valued  $C^\infty$ smooth functions with compact support in $\Omega$.

When $n=2$, which is the main (but not the sole) concern of this paper,
it is often convenient to use
the complex differential operators:
\begin{equation}
\label{defderC}
\partial=\frac{1}{2}(\partial_{x_1}-i\partial_{x_2}),\qquad
\bar\partial=\frac{1}{2}(\partial_{x_1}+i\partial_{x_2}),
\end{equation}
so that $df=\partial fdz+\bar\partial fd\bar z$. When $f$ is holomorphic:
$\bar\partial f=0$, we also write $f'$ instead of 
$\partial f=df/dz$. 

We put
$W^{1,p}_{loc}(\Omega)$ for the space of functions
whose restriction to any relatively compact open subset 
$\Omega_0$ of $\Omega$ lies in $W^{1,p}(\Omega_0)$. 
The space $W^{2,p}(\Omega)$ is comprised of
$L^p$-functions whose distributional derivatives of the first 
order lie in $W^{1,p}(\Omega)$, with norm
$\|f\|_{W^{2,p}(\Omega)}^p=\|f\|^p_{L^p(\Omega)}+\sum_j\|\partial_{x_j}f\|_{W^{1,p}(\Omega)}^p$.
The definition of $W^{2,p}_{loc}(\Omega)$ parallels that of 
$W^{1,p}_{loc}(\Omega)$.

For emphasis, we use at places 
a subscript ``$\RR$'', as in $W^{1,p}_\RR(\Omega)$, to
single out the real subspace of real-valued functions. The same symbol
({\it e.g.} ``$C$'') is used many times to mean different constants.
We write $A\sim B$ to abbreviate $C A\leq B\leq C'A$, where $C,C'$ 
are constants. 

If $n=1$, then $W^{1,p}(\Omega)$  is just the space of 
locally  absolutely continuous  functions
with derivative  in $L^p(\Omega)$. The corresponding 
characterization when $n>1$ is more  subtle 
\cite[Thm. 2.1.4]{ziemer}, but in any case 
$W^{1,\infty}(\Omega)$  identifies with 
Lipschitz-continuous functions on $\Omega$ \cite[Sec. V.6.2]{stein}.

An open set $\Omega\subset\RR^n$ is called Lipschitz if, in a
neighborhood
of each boundary point, it is
isometric to the epigraph of a 
Lipschitz function  \cite[Def. 1.2.1.1]{gris}.
When $\Omega$ is bounded and Lipschitz, each member of $W^{1,p}(\Omega)$ is 
the restriction to $\Omega$ of a function in $W^{1,p}(\RR^n)$
(the extension theorem \cite[Ch. VI, Thm 5]{stein}), and
the space of restrictions
$(\mathcal{D}(\RR^n))_{|\Omega}$ is dense in
$W^{1,p}(\Omega)$ for $1\leq p<\infty$ \cite[Thm 3.22]{Adams}.
Here and below, the subscript ``${|E}$''
indicates restriction to a set $E$.
If moreover $p>n$ then $W^{1,p}(\Omega)$ embeds continuously
in the space of  H\"older-continuous functions on $\Omega$ with 
exponent  $1-n/p$;  when $p=n$ such an embedding holds  in
every $L^\ell(\Omega)$, $1\leq \ell<\infty$, and if $p<n$ then
$W^{1,p}(\Omega)$ embeds continuously in $L^{p_*}$ with  $p_*=np/(n-p)$
(the Sobolev embedding theorem \cite[Thms 4.12, 4.39]{Adams}).
In addition, for $p\leq n$ and $\ell<p_*$ ($p_*=\infty$ if $p=n$), 
the previous  { embeddings are} compact 
(the Rellich-Kondrachov theorem \cite[Thm 6.3]{Adams}).

Also, a distribution $g$ on  $\Omega$ whose first 
 derivatives lie  in $L^p(\Omega)$  does belong to
 $W^{1,p}(\Omega)$ \cite[Thm 6.74]{Demengel}\footnote{The proof given 
 there for bounded $C^1$-smooth $\Omega$ carries over to 
 the Lipschitz case.}, and  if $\Omega$ is connected while 
$E\subset\Omega$ is such that $m_n(E)>0$,
then
 \begin{equation}
 \label{estfonc}
 \left\|g-g_E\right\|_{L^p(\Omega)}\leq C \|\,\nabla g\,\|_{L^p(\Omega)},\quad
 \,\,\,\mbox{where}\,\,\, g_E:=\frac{1}{m_n(E)}
 \int_E g\,dm_n,
\end{equation}
 for some   $C=C(p,\Omega,E)$ (the Poincar\'e inequality, apply
\cite[Thm 4.2.1]{ziemer} with $L(u)=u_E$).
The Sobolev embedding theorem entails that $W^{1,p}(\Omega)$ is an algebra for
$p>n$ \cite[Thm 4.39]{Adams}, in particular if $f\in W^{1,p}(\Omega)$ and
$F$ is entire then $F(f)\in W^{1,p}(\Omega)$ with norm bounded in terms of 
$\Omega$, $p$, $F$, and $\|f\|_{W^{1,p}(\Omega)}$.

For $1<p<\infty$, the space $W^{\theta,p}(\Omega)$ of fractional order 
$\theta\in(0,1)$ consists of those $f\in L^p(\Omega)$ for which
\begin{equation}
\label{defSobfrac}
\|f\|^p_{W^{\theta,p}(\Omega)}=\|f\|^p_{L^p(\Omega)}+
\int_\Omega\int_\Omega \frac{|f(x)-f(y)|^p}{|x-y|^{n+\theta p}}dm_n(x)dm_n(y)
<\infty.
\end{equation}
The space $W^{1+\theta,p}(\Omega)$ is comprised of $f\in L^p(\Omega)$
whose derivatives of the first order lie in $W^{\theta,p}(\Omega)$, with norm
$\|f\|_{W^{1+\theta,p}(\Omega)}^p=\|f\|_{L^p(\Omega)}^p+\sum_j\|
\partial_{x_j}f\|_{W^{\theta,p}(\Omega)}^p$.

When $\Omega$ is bounded and  Lipschitz, $W^{\theta,p}(\Omega)$ may also
be defined {\it via}  real interpolation between 
$L^p(\Omega)$  and $W^{1,p}(\Omega)$ where it corresponds
to the Besov space $B^{\theta,p,p}(\Omega)$; that is, using standard 
notation for the interpolation functor, it holds that
$W^{\theta,p}(\Omega)=[L^p(\Omega),W^{1,p}(\Omega)]_{\theta,p}$, see 
\cite[Sec. 7.32 \& Thm 7.47]{Adams}. 

A slightly different, but equivalent 
interpolation method is
that of trace spaces of J.-L. Lions
\cite[Ch. 7]{Adams1}. If $\textrm{d}(x,\partial\Omega)$ denotes 
Euclidean distance from $x\in\RR^n$ to the boundary of $\Omega$, there is 
$C=C(\Omega,\theta,p)$ such that for all $f \in L^p(\Omega)$ with $|\nabla f| \in L^p_{loc}(\Omega)$, 
\begin{equation}
\label{inegfracdb}
\|f\|_{W^{\theta,p}(\Omega)}\leq C
\Bigl(\|\textrm{d}(.,\partial\Omega)^{1-\theta}\,\nabla f\|_{L^p(\Omega)}+
 \|f\|_{L^p(\Omega)}\Bigr).
\end{equation}
In fact, \cite[Thm 4.1]{JKinhom} asserts that the left and right hand sides of
\eqref{inegfracdb} are equivalent when $f$ is harmonic 
(with constants depending only on $\Omega$), and one can check 
that the portion of proof yielding \eqref{inegfracdb} 
(which rests on trace space interpolation) does not depend on harmonicity.

Recall the basic property of interpolation: if $A$ is 
linear and continuous both $X\to X'$ and $Y\to Y'$ where $(X,X')$ and 
$(Y,Y')$ are interpolation pairs of Banach spaces, then
$A$ is continuous $[X,Y]_{\theta,p}\to [X',Y']_{\theta,p}$
\cite[Thm 7.23]{Adams}. 
From this, a fractional version of the Sobolev embedding 
theorem is easily   obtained
\cite[Cor. 4.5.3]{Demengel}. 
Namely, if $\theta p>n$ then $W^{\theta,p}(\Omega)$ embeds continuously
in H\"older-continuous functions with 
exponent  $\theta-n/p$; if $\theta p=n$, such an embedding holds  in
$L^\ell(\Omega)$ for $1\leq \ell<\infty$; if $\theta p<n$, then
$W^{\theta,p}(\Omega)$ embeds continuously in $L^{p^*}$ with  
$p^*=np/(n-\theta p)$.

When $\Omega$ is Lipschitz and bounded, its boundary 
$\partial\Omega$ is a compact $(n-1)$-dimensional Lipschitz manifold 
on which $L^p(\partial\Omega)$, $W^{1,p}(\partial\Omega)$, and
$W^{\theta,p}(\partial\Omega)$ are defined as before, only with area
measure  $d\Sigma$ instead of $dm_n$ and Lipschitz-continuous test functions
rather than smooth ones \cite[Sec. 1.3.3]{gris}.  
For $1<p<\infty$,
each $f \in W^{1,p}(\Omega)$ has a trace on 
$\partial \Omega$, denoted again by $f$ or sometimes
$\tr_{\partial \Omega} \ f$ for emphasis, 
whose pointwise definition $\Sigma$-a.e.
rests on the extension theorem 
and the fact that non-Lebesgue points of $f$
have 1--Hausdorff 
measure 
zero \cite[Ch. 4, Rmk 4.4.5]{ziemer}. In particular, 
$\tr_{\partial \Omega} \ f$ coincides with the limit of $f$ at points
of $\partial\Omega$ where this limit exists.
The function $\tr_{\partial \Omega} \ f$  lies  in 
$W^{1-1/p,p}(\partial\Omega)$
\cite[{Thm} 7.47]{Adams}, \cite[Sec. 1.3.3]{gris},
and the trace operator defines a continuous
surjection from
$W^{1,p}(\Omega)$ onto  $W^{1-1/p,p}(\partial\Omega)$ with
continuous right 
inverse  \cite[Thm 1.5.1.3]{gris}.
The subspace $W^{1,p}_0(\Omega)$ of functions whose trace is 
zero coincides with the closure of $\mathcal{D}(\Omega)$ in $W^{1,p}(\Omega)$
\cite[Cor. 1.5.1.6]{gris}.
If $\Omega$ is connected,
a variant of the Poincar\'e inequality involving the trace is: 
for  $p>1$ and
$\Gamma\subset\partial\Omega$ a subset of strictly
positive measure $\Sigma(\Gamma)>0$,
there is 
$C>0$ depending only on $p$, $\Omega$ and $\Gamma$ such that for all $g \in W^{1,p}(\Omega)$
\begin{equation}
\label{estfoncb}
\Bigl\|g-\int_\Gamma\tr_{\partial\Omega} \,g \, d \Sigma\Bigr\|_{L^p(\Omega)}\le C
\|\nabla g\|_{L^p(\Omega)}.
\end{equation}
This follows from the continuity of the trace operator, the 
Rellich--Kondrachov theorem  and \cite[Lem. 4.1.3]{ziemer}.

We need mention Sobolev spaces of negative order in connection with 
duality of trace spaces: if $1<p<\infty$ and $1/p+1/q=1$ then,
since $W^{1/q,p}(\partial\Omega)$ embeds in $L^p(\partial\Omega)$,
each $g\in L^q(\partial\Omega)$ gives rise via
$h\mapsto\int_{\partial\Omega} g\bar hd\Sigma$ 
to a member of
$(W^{1/q,p}(\partial\Omega))'$,  the dual space of 
$W^{1/q,p}(\partial\Omega)$. As $W^{1/q,p}(\partial\Omega)$ is reflexive
(for it is uniformly convex), we see as in 
\cite[Sec. 3.13, 3.14]{Adams} 
that the completion 
$W^{-1/q,q}(\partial\Omega)$ of  $L^q(\partial\Omega)$ with respect to the norm
\[\|g\|_{W^{-1/q,q}(\partial\Omega)}:=\sup_{\|h\|_{W^{1/q,p}(\partial\Omega)}=1}\,\,\left|\int_{\partial\Omega} g\bar h\,d\Sigma\right|\] 
can be identified with $(W^{1/q,p}(\partial\Omega))'$. 
This we use when $p=q=2$ only.

\subsection{Non tangential  maximal function}
\label{ssec:ntmf}
For $\xi\in\partial\Omega$,
each $\alpha>1$ defines a  nontangential 
region of approach to $\xi$ from $\Omega$ given by
\begin{equation}
\label{rNT}
R^{\Omega}_\alpha(\xi)=\{x\in\Omega:\ |x-\xi|<\alpha \,\text{d}
(x,\partial\Omega)\}.
\end{equation}
When $\Omega$ is Lipschitz and bounded,  $R^{\Omega}_\alpha(\xi)$
contains a nonempty open truncated cone with vertex $\xi$, 
whose aperture and height are independent 
of $\xi$ \cite[Thm 1.2.2.2]{gris}.
Subsequently,
whenever $h$ is $\RR^k$-valued on $\Omega$, we define 
its nontangential maximal function (associated with $\alpha$) to be 
\begin{equation}
\label{defNT}
\mathcal{M}_\alpha h(\xi)=\sup_{x\in R^\Omega_\alpha(\xi)}|h(x)|,
\qquad\xi\in\partial\Omega,
\end{equation}
which is well-defined with values in $[0,+\infty]$.
Also, we say that $h$ defined on $\Omega$ converges nontangentially 
to $a$ at $\xi\in\partial\Omega$ if, for every $\alpha>1$,
\begin{equation}
\label{defntc}
\lim_{x\to\xi,\,x\in R^\Omega_\alpha(\xi)}h(x)=a.
\end{equation}

\subsection{Planar case}
\label{sec:pc}

In dimension $n=2$, 
$\partial\Omega$ is a curve and tangential differentiation produces
a total derivative. This makes for specific notation 
as follows. A simply connected Lipschitz domain $\Omega\subset\RR^2$  has
a rectifiable Jordan curve as boundary and
we  write $\Lambda$ (instead of $\Sigma$) for arclength measure on
$\partial\Omega$. 
We let $\tau$ and $n$ respectively indicate the unit tangent
and  (outwards pointing)  normal vector fields on $\partial\Omega$,
which are well defined in 
$L^\infty(\partial\Omega)\times L^\infty(\partial\Omega)$ 
\cite[Sec. 1.5.1]{gris}.
Here, $\tau$ is oriented so that $(n,\tau)$ is a positive frame $\Lambda$-a.e.

By what we said before,
$W^{1,p}(\partial\Omega)$  consists of
absolutely continuous functions with respect to $\Lambda$ 
whose derivative lies in $L^p(\partial\Omega)$.
We shall write $\partial_\tau h$ instead of
$d h/d\Lambda$.
If $\varphi$ is smooth on a neighborhood of
$\partial\Omega$ in $\RR^2$, then the restriction
$\psi=\varphi_{|\partial\Omega}$
belongs to $W^{1,\infty}(\partial\Omega)$ and 
$\partial_\tau\psi=\nabla\varphi.\tau$.
Using duality, one can extend the definition of tangential derivative 
to less smooth classes of functions, but  at this point we restrict the 
discussion to $p=2$ 
which is enough for our purposes\footnote{Appealing
to \cite[Ch. II, Thm 1.1]{Lions60} instead of 
\cite[Ch. I, Thm 6.2]{lions_magenes}, the same reasoning 
shows that $\partial_\tau$ is
continuous  $W^{1/q,p}(\partial\Omega)\to W^{-1/p,p}(\partial\Omega)$ 
for $1<p<\infty$, $1/p+1/q=1$.}.
For $f\in L^2(\partial\Omega)$,  
define
$\partial_\tau f\in (W^{1,2}(\partial\Omega))'$ to be the linear form
$h\mapsto-\int_{\partial\Omega}f\partial_\tau h d\Lambda$, 
$h\in W^{1,2}(\partial\Omega)$. This generalizes the  previous definition 
of $\partial_\tau$ when  $f\in W^{1,2}(\partial\Omega)$, for in this case
integration by parts shows that
the linear form just mentioned extends to a member of 
$(L^2(\partial\Omega))'\sim L^2(\partial\Omega)$ which is just 
$\partial_\tau f$ in the former  sense. Thus, by interpolation, 
$\partial_\tau$ is continuous from
$W^{1/2,2}(\partial\Omega)$
into the space $(W^{1/2,2}(\partial\Omega))'\sim W^{-1/2,2}(\partial\Omega)$. Indeed, from \cite[Ch.I, Thm 6.2]{lions_magenes}:
\[
[(L^2(\partial\Omega))', (W^{1,2}(\partial\Omega))']_{1/2,2}
=
([W^{1,2}(\partial\Omega),L^2(\partial\Omega)]_{1/2,2})'=
(W^{1/2,2}(\partial\Omega))' \, .\]

\section{Conductivity equation and Robin inverse problem}
\label{sec:cond}

{
In Section \ref{secc:cond} we {{introduce the conductivity equation
under study}}.
Sections \ref{secc:32} and \ref{sec:sp1} are dedicated to the associated forward Neumann and Robin problems.
Section \ref{sec:ur10} concerns the inverse Robin problem.
}
\subsection{The conductivity equation}
\label{secc:cond}
The conductivity equation with unknown real-valued 
function $u$ is
\begin{equation}
\label{forcond}
\nabla \cdot \left( \sigma \, \nabla u \right) =0\,,
\end{equation}
where ``$\nabla\cdot X$'' means  
``divergence of the vector field $X$''. Except in Section \ref{anisotropic}, we assume that the
conductivity $\sigma$ is  a real-valued function on
a bounded Lipschitz domain $\Omega\subset\RR^n$ satisfying
\begin{align}
\label{isot}
\quad&
\quad\sigma \in W_\RR^{1,r}(\Omega),\quad r>n,\\
\label{ellip}
\quad&\quad 0 < c \leq \sigma \leq 1/c < + \infty
\quad\text{for some constant}\ c.
\end{align}
The fact that $\sigma$ is real means that the 
conduction is isotropic. Condition \eqref{ellip} above means that \eqref{forcond} is
strictly elliptic. Condition \eqref{isot} is less restrictive than 
Lipschitz-regularity, but still it implies some H\"older-smoothness.  
{{Note that,}} since $r>n$, the space $W^{1,r}(\Omega)$ consists
of multipliers on $W^{1,2}(\Omega)$, see
\cite{Zolesio} or \cite[Thm 1.4.4.2]{gris}.

As \eqref{isot} and \eqref{ellip} together imply that 
$1/\sigma\in W^{1,r}_\RR(\Omega)$,  our assumptions
are thus to the effect  that multiplication by (the restriction to
$\partial\Omega$ of) $1/\sigma$ 
is an isomorphism on $W_\RR^{1/2,2}(\partial\Omega)$.
By duality, it follows that multiplication by $1/\sigma$ is an isomorphism
on $W_\RR^{-1/2,2}(\partial\Omega)$. This entails that
each $u\in W_\RR^{1,2}(\Omega)$ 
solving for \eqref{forcond} has 
a well-defined normal derivative  on 
$\partial\Omega$,  denoted by 
$\partial_n u\in W^{-1/2,2}_\RR(\partial\Omega)$. 
The standard definition 
is the weak one: if $J$ designates a right inverse to 
the trace operator 
$W_\RR^{1,2}(\Omega)\to W_\RR^{1/2,2}(\partial\Omega)$
and $\langle\ ,\ \rangle$ the duality pairing on
$W^{-1/2,2}_\RR(\partial\Omega)\times W_\RR^{1/2,2}(\partial\Omega)$,  then
$h\mapsto \int_\Omega \sigma \nabla u.\nabla(J(h)) \, d \, m_n$
is a continuous linear 
form on $W_\RR^{1/2,2}(\partial\Omega)$ which  can be represented uniquely 
as $\langle \phi, h\rangle$
for some $\phi\in W_\RR^{-1/2,2}(\partial\Omega)$. Since division by 
$\sigma$
is an isomorphism of $W_\RR^{-1/2,2}(\partial\Omega)$, we may set 
$\partial_n u=\phi/\sigma\in W_\RR^{-1/2,2}(\partial\Omega)$ and 
then it holds that
\begin{equation}
\label{defndw}
\langle \sigma\partial_nu\,, \,\tr_{\partial\Omega}\psi\rangle=
\int_\Omega \sigma \nabla u.\nabla \psi\,dm_n,\qquad \psi\in W^{1,2}_\RR(\Omega).
\end{equation}
Indeed, \eqref{defndw} holds by construction when  $\psi\in \text{Ran} \,  J$,
hence it is enough to check it when $\psi \in W^{1,2}_{0,\RR}(\Omega)$ 
in order to get it for all $\psi\in W_\RR^{1,2}(\Omega)$. 
By density, we are left to prove that 
$\int_\Omega \sigma\nabla u.\nabla\psi\,dm_n=0$ whenever
$\psi\in\mathcal{D}_\RR(\Omega)$ which is nothing but
the distributional meaning of \eqref{forcond}. Comparing \eqref{forcond} and 
\eqref{defndw} with the classical Green formula,  it is natural to 
call $\partial_n u$ the (exterior) normal derivative of $u$ on 
$\partial\Omega$.
Checking \eqref{defndw} against $ \psi\equiv1$, we observe in particular that
\begin{equation}
\label{compflux}
\langle \partial_n u\,,\,\sigma\rangle=0.
\end{equation}

\subsection{The Neumann problem}
\label{secc:32}
The Neumann problem in $W^{1,2}(\Omega)$ for the conductivity
equation \eqref{forcond}
is: given 
$g\in W_\RR^{-1/2,2}(\partial\Omega)$ such that 
$\langle g\,,\,\sigma\rangle=0$, to find $u\in W^{1,2}_\RR(\Omega)$
such that
\be \left\{\begin{array}{clc} 
&\nabla \cdot \left( \sigma \, \nabla u \right) =0\,\,\, {\rm in}\,\,\,\Omega,
&\\
&\displaystyle  \partial_n u= g\,\,\, {\rm on}\,\,\,\partial\Omega.&
\end{array} \right. 
\label{Neumann} 
\ee
Note that the vanishing of $\langle g\,,\,\sigma\rangle$
is necessary by \eqref{compflux}.
A solution to \eqref{Neumann} exists and is unique up to an 
additive constant. To check this well-known fact, simply observe that
$f\mapsto \langle \sigma g,\tr_{\partial\Omega} f\rangle$ is a continuous 
linear form on $W^{1,2}_\RR(\Omega)/\RR$ (the quotient space of 
$W^{1,2}_\RR(\Omega)$ modulo constants), a Hilbert norm on which is 
given by $\|\nabla f\|_{L^2(\Omega)}$ in view of  \eqref{estfonc}. 
As $\|\sigma^{1/2}\nabla f\|_{L^2(\Omega)}$ is an equivalent norm by 
\eqref{ellip}, we see upon denoting by $\check{u}\in W^{1,2}_\RR(\Omega)/\RR$ 
the equivalence class of 
$u\in W_\RR^{1,2}(\Omega)$ that
there is a unique $\check{u}$ 
to meet \eqref{defndw} with $\partial_n u$ replaced by $g$,
thanks to the Lax-Milgram theorem \cite[Cor. V.8]{Brezis}. 
As pointed out earlier,
this is equivalent to $u$ solving \eqref{Neumann}.
Such a  $u$ is called an energy solution to the Neumann problem.
\subsection{The forward Robin problem} 
\label{sec:sp1}
The forward Robin problem is an implicit variation of the Neumann problem
where the solution to \eqref{forcond} and
its normal derivative have to satisfy  an affine relation
with functional coefficients on the boundary. In particular, the normal 
derivative is sought to be a function rather than a distribution on
$\partial\Omega$. Below we consider a rather simple form of the problem, 
arising naturally in the setting of non-destructive control, where
the affine relation has $L^2$ right-hand side and bounded coefficient.
More general versions with  right-hand side in $L^p(\partial\Omega)$, 
$p\in(1,2]$, are studied in
\cite{lanzani-shen}.

Throughout we assume that $\partial \Omega$ is partitioned into 
measurable subsets $\Gamma$ and $\Gamma_0$ of strictly
positive arclength:
\begin{equation}
\label{partbord}
\partial \Omega={\Gamma} \cup {\Gamma_0}, \quad
\Gamma \cap \Gamma_0=\emptyset, \quad \Sigma(\Gamma) >0, \quad \Sigma(\Gamma_0) >0.
\end{equation}
We put for simplicity
\begin{equation}
\label{defL+}
L^\infty_+(\Gamma):=\{\la \in L_\RR^\infty(\Gamma)\, ,\  \la \geq 0 \mbox{ a.e. on } \Gamma \, , \ \la \not \equiv 0 \}.
\end{equation} 

Given $\la \in L_+^\infty(\Gamma)$ and
$g\in L^2(\Gamma_0)$, the forward Robin problem consists in seeking
$u \in W^{1,2}(\Omega)$ such that
\be \left\{\begin{array}{clc} 
&\nabla \cdot \left( \sigma \, \nabla u \right) =0\,\,\, {\rm in}\,\,\,\Omega,&\\
&\displaystyle \partial_n u= g\,\,\, {\rm on}\,\,\,\Gamma_0,&\\
&\displaystyle  \partial_n u + \la u= 0\,\,\, {\rm on}\,\,\,\Gamma.& 
\end{array} \right. 
\label{forward10} 
\ee
As $\tr_{\partial\Omega}u\in W^{1/2,2}_\RR(\partial\Omega)\subset L_\RR^2(\partial\Omega)$, boundary conditions make sense in that 
$g$ concatenated with 
$(-\lambda\tr_{\partial\Omega}u)_{|\Gamma}$  defines a member of 
$L^2_\RR(\partial\Omega)\subset W^{-1/2,2}_\RR(\partial\Omega)$. 

Replacing $g$ and $\la$ by $g/\sigma$ and $\la/\sigma$ respectively,
which is possible by \eqref{ellip}, 
solving \eqref{forward10}  is 
tantamount to obtain $u\in W^{1,2}_\RR(\Omega)$ satisfying
\be \left\{\begin{array}{clc} 
&\nabla \cdot \left( \sigma \, \nabla u \right) =0\,\,\, {\rm in}\,\,\,\Omega&\\
&\displaystyle \sigma \, \partial_n u= g\,\,\, {\rm on}\,\,\,\Gamma_0&\\
&\displaystyle \sigma \, \partial_n u + \la u= 0\,\,\, {\rm on}\,\,\,\Gamma.& 
\end{array} \right. 
\label{forward1} 
\ee
In view of \eqref{defndw}, problem (\ref{forward1}) is equivalent to
the following  weak formulation: to find $u$ in $W^{1,2}(\Omega)$ such that 
\be
\label{weak1}
\int_\Omega \sigma \, \nabla u \cdot \nabla \psi\,dm_n + \int_{\Gamma}\la u\psi\,d \Sigma = \int_{\Gamma_0}g\psi\,d\Sigma,\qquad \psi\in W^{1,2}_\RR(\Omega).
\ee
{{As soon as $\sigma\in L^\infty_\RR(\Omega)$}}, well-posedness of problem (\ref{weak1}), that is, existence and uniqueness of
a solution $u\in W^{1,2}(\Omega)$,  follows at once from 
the Lax-Mil\-gram theorem and 
Lemma \ref{lem:equivnorm1} below. 
Further, as a consequence of 
\eqref{weak1}, it holds that
\[
\int_{\Gamma}\la u\,d \Sigma = \int_{\Gamma_0}g\,d\Sigma \, .
\]
\begin{lemma}
\label{lem:equivnorm1}
Let $\Omega\subset\RR^n$ be a bounded Lipschitz domain, 
$\sigma\in L^\infty_\RR(\Omega)$ satisfy \eqref{ellip},
and $\lambda\in L_+^\infty(\Gamma)$
for some
$\Gamma\subset\partial\Omega$ such that $\Sigma(\Gamma)>0$.
Then, 
\[
u\mapsto\left(\int_\Omega \sigma \, \left|\nabla u\right|^2\,dm_n + \int_{\Gamma} \la \, u^2\,d \Sigma\right)^{1/2}\]
is an equivalent norm on $W^{1,2}(\Omega)$.
\end{lemma}
\begin{proof}
We must show that there exist two constants $c,C>0$ such that
\[c\,\left\|\psi\right\|^2_{W^{1,2}(\Omega)} \leq \int_\Omega \sigma \, \left|\nabla \psi\right|^2\,dm_n + \int_{\Gamma} \la \, \psi^2\,d \Sigma 
\leq C\,\left\|\psi\right\|^2_{W^{1,2}(\Omega)} \, , \quad  
\psi \in W_\RR^{1,2}(\Omega).\]
The right inequality follows from the boundedness of 
$\sigma,\lambda$, together with the continuity of the trace operator and 
the embedding $W^{1/2,2}(\partial\Omega)\to L^2(\partial\Omega)$.
To prove the left inequality  we can replace $\Gamma$ by a subset on which 
$\lambda\geq\varepsilon>0$, and then the result drops out from 
\eqref{estfoncb}, the Schwarz inequality, and the fact that $\sigma$ is bounded away from $0$ by \eqref{ellip}.
\end{proof}
\subsection{The inverse Robin problem}
\label{sec:ur10}

Associated to the forward Robin problem
\eqref{forward1} is the inverse Robin problem,
which consists in finding the 
unknown impedance $\la$ in $L_+^\infty(\Gamma)$ 
from the knowledge of $u$ and $g$  on $\Gamma_0$. Note that a solution 
$u$ to \eqref{forward1} uniquely exists in $W^{1,2}(\Omega)$,
as was pointed out before Lemma \ref{lem:equivnorm1} above. In the setting of 
nondestructive control,
$\Gamma_0$  represents that part of the boundary $\partial\Omega$ which is 
accessible to pointwise  {{measurement or imposition}} 
of $u$ and $\partial_n u$. 

{{\emph{In this work, we consider the uniqueness issue as to whether $\lambda$ 
is uniquely determined by $g$ and $u_{|\Gamma_0}$ when $\Omega$ is a bounded contractible
Lipschitz domain.}}} For general partitions 
of the boundary like \eqref{partbord}, 
it will turn out that the answer is ``yes''
when $n=2$ and ``no'' when $n\geq3$. Pointing out
this structural difference 
between the planar and  the higher dimensional cases is the main purpose 
of the present { article}. 

\section{Uniqueness results}
\label{sec:ur00} 

{ The two uniqueness theorems in Section \ref{sec:ur1} are the main  results of this work. Section \ref{ce} provides a counterexample in dimension 3.}

\subsection{Inverse Robin problem in dimension 2: uniqueness results}
\label{sec:ur1} 

In this section we investigate the planar case: $\Omega\subset\RR^2$, 
in particular it is understood throughout that
$n=2$ in \eqref{isot} and we write $\Lambda$ instead of $\Sigma$
in \eqref{partbord}.

\begin{theorem}
Assume that $\Omega\subset\RR^2$ is a bounded simply connected  Lipschitz 
domain and that \eqref{partbord} holds.
Let $\sigma$ satisfy  \eqref{isot}-\eqref{ellip}
and $g \in L^2(\Gamma_0)$ {{be}} such that $g \not\equiv 0$. Suppose $\la_1,\la_2 \in L^\infty_+(\Gamma)$ 
are such that the corresponding solutions $u_1,u_2 \in W_\RR^{1,2}(\Omega)$ 
to problem (\ref{forward1}) satisfy $u_{1|_{\Gamma_0}}=u_{2|_{\Gamma_0}}$.
Then $\la_1=\la_2$.
\label{main1}
\end{theorem}
Theorem \ref{main1} will be a
consequence of the following unique continuation result which is
proved in Section \ref{sec:proof1}.
\begin{theorem}
\label{cauchy1}
Assume that $\Omega\subset\RR^2$ is a bounded simply connected Lipschitz 
domain and that $\sigma$ satisfies  \eqref{isot}-\eqref{ellip}.
Let $u \in W^{1,2}_\RR(\Omega)$ be a solution to
\eqref{forcond} 
in $\Omega$ such that $\partial_n u\in L^2(\partial\Omega)$. 
If both $u$ and $\partial_n u$ 
vanish on a subset 
$\gamma \subset \partial \Omega$ of {{strictly}} positive measure, 
then $u\equiv0$  in $\Omega$.
\end{theorem}
\begin{proof} ({\it Theorem \ref{main1}}) 
{By assumption, $u_1$ and $u_2$ have the same Cauchy data on 
$\Gamma_0\subset \partial \Omega$ with $\Lambda(\Gamma_0) > 0$,
so Theorem \ref{cauchy1} implies that $u_1 \equiv u_2$ in 
$\overline{\Omega}$
whence $(\la_1-\la_2)\, u_1=0$ on $\Gamma$ by the
Robin boundary condition.
Assume for a contradiction that $\la_1 \neq \la_2$ {a.e.} on $\Gamma$. 
Then, there exists a subset $\gamma \subset \Gamma$, $\Lambda(\gamma) >0$, such that
$\la_1 - \la_2 \neq 0$ on $\gamma$, and of necessity 
$u_1$ vanishes on $\gamma$ by what precedes. 
In turn $\partial_n u_1= - \la_1 \, u_1/\sigma$ vanishes identically on 
$\gamma$, therefore Theorem \ref{cauchy1} implies  that
$u_1\equiv0$ in $\Omega$. Consequently $\partial_n u_1= 0$ a.e. on 
$\partial \Omega$, thereby contradicting the assumption 
that $g \not\equiv 0$ in $\Gamma_0$.}
\end{proof}

The proof of Theorem \ref{cauchy1}  (see Section \ref{sec:proof1} and Theorem \ref{normeq3/2}) ultimately rests on the fact that,
in dimension 2, a harmonic gradient ({\it i.e.} the conjugate 
of a holomorphic function if we identify $\CC$ with $\RR^2$)  
which has nontangential limit zero on a subset of 
$\partial\Omega$ of positive measure is identically zero (see Section \ref{secH00}).

This is no longer
true in higher dimension, as illustrated in the next section.

\subsection{{{Examples of  non uniqueness }}  in  higher dimension}
\label{ce}

An initial example was constructed in  \cite[Thm 1]{Wolff1} of a 
nonconstant harmonic function on a half space in $\RR^3$,
with H\"older-continuous  derivatives up to the boundary, whose gradient vanishes on a boundary
set $E$ 
with $m_2(E)>0$, see also \cite{AK}.
In \cite{bourgain_wolff}, this construction was refined to the effect
that there is a nonzero harmonic 
function on a half space, $C^1$-smooth up to the boundary, that vanishes 
together  with its normal derivative on a boundary set $E$ with $m_2(E)>0$. 
In fact, such examples can be constructed on any open subset of
$\RR^n$, $n\geq3$, whose boundary is a $C^{1,\varepsilon}$ manifold
\cite{WW}. This shows that Theorem \ref{cauchy1}  does not hold  in 
dimension strictly bigger than 2, and
casts doubt on whether an analog to Theorem \ref{main1} can hold 
in higher dimension. Indeed, the example below shows that it cannot, 
already for harmonic functions on 
smooth domains. 

Hereafter, we denote by $\mathbb{B}^3\subset\RR^3$ the 
open unit ball and by $\mathbb{S}^2$ the boundary sphere (recall the  definition \eqref{defL+} of $L^\infty_+(\Gamma)$).
\begin{example}
\label{rmk2ex}
{{Let $u$ be a nonzero harmonic function in $\mathbb{B}^3$, of class $C^1$ on $\overline{\mathbb{B}^3}$, 
such that $u_{|E}=(\partial_n u)_{|E}=0$ where $E\subset\mathbb{S}^2$, with $\Sigma(E)>0$.
If in problem (\ref{forward10}) we set:
\[\Gamma_0=\{\xi\in\mathbb{S}^2,\ u^2(\xi)+\partial_nu^2(\xi)\neq0\},
\qquad
g=\partial_nu_{|\Gamma_0},\] 
then 
$\Gamma:=\mathbb{S}^2\setminus\Gamma_0$ contains $E$ hence it}}  has strictly positive 
$\Sigma$-measure,
but clearly $\lambda$ can be arbitrary in $L^\infty_+(\Gamma)$ 
since $u_{|\Gamma}=\partial_nu_{|\Gamma}\equiv0$.
\end{example}
Example \ref{rmk2ex} shows that a  solution to \eqref{forward10}
may be associated to all Robin functions. This is an extreme example of non
uniqueness which,  however, is not fully satisfactory in that it is highly non
generic and will be destroyed by small  perturbations of the Neumann boundary
data $g$ on $\Gamma_0$. 
The theorem below gives another example of non uniqueness which is easily seen
to be stable
under $L^p(\Gamma_0)$-small perturbations of  $g$, for $p >2$.
\begin{theorem}
\label{CEBW}
Set $\sigma\equiv1$ on $\mathbb{B}^3$. Then,
there is a partition $\mathbb{S}^2={\Gamma} \cup {\Gamma_0}$ of the form
\eqref{partbord}, along with functions $g\in L_\RR^2(\Gamma_0)$
and $\lambda_1\neq \lambda_2\in L^\infty_+(\Gamma)$  
such that the corresponding solutions 
$u_1$, $u_2$ to
\eqref{forward10} on $\mathbb{B}^3$, though distinct, satisfy $(u_1)_{|\Gamma_0}=
(u_2)_{|\Gamma_0}$.
\end{theorem}
\begin{proof} 
Let $\Gamma_0\subset\mathbb{S}^2$, 
$\Sigma(\Gamma_0)>0$, have the property that 
there is a nonzero harmonic function $u$ in $\mathbb{B}^3$,
of class $C^1$ on $\overline{\mathbb{B}^3}$, 
with $u_{|\Gamma_0}=(\partial_n u)_{|\Gamma_0}=0$. Such a $\Gamma_0$ exists by  { {\cite{bourgain_wolff}}}. Let 
$h \in L^\infty_\RR(\mathbb{S}^2)$ be such that $0<c<-h<C$ on 
$\Gamma=\mathbb{S}^2\setminus\Gamma_0$
for some constants $c,C$, and moreover $\int_{\mathbb{S}^2}hd\Sigma=0$.
In addition, we pick $c$ large enough that $h<-|\partial_n u|-1$ on 
$\Gamma$.
Let $v$ be a solution to the Neumann problem \eqref{Neumann} where
$\sigma\equiv1$, $\Omega=\mathbb{B}^3$, and
$\partial_n v=h$. On the sphere, the ``Riesz tranform'' mapping the normal derivative 
of a harmonic function $w$ in $\mathbb{B}^3$ 
to its tangential gradient vector field 
is continuous in $L^p$-norm for $1<p<\infty$; this follows easily by dominated 
convergence from the fact that, for each $\alpha>1$,
$\|\mathcal{M}_\alpha \nabla w\|_{L^p(\mathbb{S}^2)}\leq
C_\alpha \, \|\partial_n w\|_{L^p(\mathbb{S}^2)}$, see \cite[Thm 2.6]{FJR}.
Therefore $v_{|\mathbb{S}^2}\in W^{1,p}_\RR(\mathbb{S}^2)$ for all 
$p\in(1,\infty)$, hence it is bounded by the Sobolev embedding theorem.
Thus, upon adding a positive constant to $v$, we may assume that 
$v>|u|+1$ on $\Gamma$ and that the function $v\partial_n u-h u$ does not identically vanish on $\Gamma$. Now, letting
$\lambda_1=-h/v$ and $\lambda_2=-(h+\partial_n u)/(u+v)$ on $\Gamma$, we 
have that $\lambda_1,\lambda_2 \in L^\infty_+(\Gamma)$, $\lambda_1 \not \equiv \lambda_2$, while the functions $u_1=v$ and $u_2=v+u$ coincides together with their normal 
derivatives on
$\Gamma_0$, as desired.
\end{proof}
{{Counterexamples similar to the one in}}  Theorem \ref{CEBW}
can be constructed in any dimension greater than 3. 
\begin{remark}
\label{rmk2}
Whenever $\sigma \in W^{1,\infty}_\RR(\Omega)$ and $\gamma$  contains an 
open subset of $\partial \Omega$, it is not difficult to deduce
from the unique continuation result in \cite{GaLi} that
the analog of Theorem \ref{cauchy1} holds for any $n\geq2$. 
However,
Example \ref{rmk2ex} shows that assuming $\Gamma_0$ open cannot rescue 
 a higher dimensional analog of Theorem  \ref{main1}.
The situation becomes  more interesting if we assume that the interiors of 
$\Gamma_0$ and $\Gamma$ fill $\mathbb{S}^2$ up to a set of $\Sigma$-measure 
zero. Then, proving or disproving the analog of Theorem \ref{main1} when
$n\geq3$ is tantamount to decide if a solution to \eqref{Neumann} that
vanishes together with its normal derivative on some $E\subset\partial\Omega$,
with $\Sigma(E)>0$, can be such that $\partial_nu/u$ is 
(essentially) bounded and nonnegative in a neighborhood of $E$ 
in $\partial\Omega$.
This question seems to be open, even for harmonic functions in a ball.
\end{remark}

\section{Hardy-Smirnov classes of holomorphic functions}
\label{secH00}

{
In Section \ref{sec:Hardy} we review Hardy spaces and conjugate 
functions on the disk,  as well as conformal 
maps onto simply connected  Lipschitz domains. 
This we use in Section \ref{sec:Smirnov} 
to discuss Smirnov spaces on Lipschitz domains, in particular of
exponent 2.  There, we bridge classical material from complex analysis with
known results from elliptic regularity theory to characterize 
Smirnov functions in terms of Sobolev smoothness (Theorem \ref{equivS2}).
Roughly speaking, Smirnov spaces consist of holomorphic functions
with Lebesgue integrable boundary values with respect to arclength, and as 
such they are basic to solve Dirichlet and Neumann problems for the Laplace 
equation in dimension 2. 
In Section \ref{Smha}, we dwell on this connection to prove well-posedness
of the Dirichlet problem with $W^{1,2}$-data which we could not 
find in the literature (Proposition \ref{Neucor}). 
This well-posedness and the fact that a nonzero 
Smirnov function cannot vanish on a boundary subset of positive measure
are fundamental to the proof of Theorem \ref{cauchy1}
in Section \ref{sec:ur00}.
}

\subsection{Hardy spaces of the disk}
\label{sec:Hardy}
We set $\DD(\xi,\rho)$ and
$\TT(\xi,\rho)$ to designate
the disk and the circle of radius $\rho$,
centered at $\xi$  in the complex plane. When $\xi=0$ we simply write
$\DD_\rho$ and $\TT_\rho$, and if 
$\rho=1$ we omit  subscripts.  
Arclength on $\TT_\rho$ will be denoted
by $m$, irrespective of $\rho$, which should cause no confusion. 
Thus, $dm(\rho e^{i\theta})=\rho d\theta$. Given a function $f$ on $\DD$
and $\rho\in[0,1)$, we write $f_\rho$ to mean the function on $\DD$ given by
$f_\rho(z)=f(\rho z)$.

For $p\in[1,\infty)$, the Hardy space $H^p$ 
consists of functions $f$ which are 
holomorphic in the unit disk and satisfy the growth condition:
\begin{equation} \label{esssuppa}
\left\Vert f\right\Vert_{H^p}= 
\sup_{0<\rho  < 1}
\Bigl(\int_{\T_\rho} \left\vert
f(\xi)\right\vert^p dm(\xi)\Bigr)^{1/p} <+\infty.
\end{equation}
The space $H^\infty$ is comprised of  bounded holomorphic functions endowed 
with
the $sup$ norm. Note that $\|f_\rho \|_{L^p(\T)}$ is non-decreasing with 
$\rho $ by subharmonicity of $|f|^p$, see \cite[Thm 17.6]{rudin},
hence the $\sup$ in \eqref{esssuppa} is really a limit as $\rho\to1^-$.

It is well-known (see \cite[Ch. 2]{duren} or \cite[Ch. 2]{garnett})
that each $f\in H^p$ has a non-tangential limit $f(\xi)$
at $m$-a.e. $\xi\in\T$, which makes for
a definition of $f$ on the unit circle.
The map $f\mapsto f_{|\TT}$ is an isometry from
$H^p$ onto the closed subspace of $L^p(\TT)$ consisting of functions 
whose Fourier coefficients of strictly negative index do vanish.
This allows us to regard $H^p$ both as a space of holomorphic functions on 
$\DD$ and as a space of $L^p$-functions on $\TT$, upon identifying $f$ with 
$f_{|\TT}$.
Every $f\in H^p$ can be represented as 
the Cauchy as well as the Poisson integral of its non-tangential limit:
\[f(z)=\frac{1}{2i\pi}\int_\TT \frac{f(\xi)}{\xi-z}d\xi,\qquad
f(z)=\frac{1}{2\pi}\int_\TT \frac{1-|z|^2}{|\xi-z|^2}f(\xi)\,dm(\xi),\qquad
z\in\DD.
\]
Hereafter, the Poisson integral of a function 
$\psi\in L^1(\TT)$ will be abbreviated as $P[\psi]$. 
If $f\in H^p$, $1\leq p<\infty$, then 
$\|(f_\rho)_{|\TT}- f_{|\TT}\|_{L^p(\T)}\to0$ as $\rho \to1^-$. 
As for the nontangential maximal function,
it holds if $p>1$ and $f\in H^p$ that, for any $\alpha>1$,
\begin{equation}
\label{ntborne}
\|{\mathcal M}_\alpha f\|_{L^p(\T)}\leq C\|f\|_{L^p(\T)}
\end{equation}
where the constant $C$ depends on $\alpha$ and $p$ 
\cite[Ch. II, Thm 3.1]{garnett}.

Clearly $H^p\subset L^p(\D)$, moreover one can see from the
Cauchy formula that if  $f\in H^p$ and $\varepsilon>0$ then the
derivative $f'$ satisfies $\|f'_\rho\|_{L^p(\TT)}\leq C(1-\rho)^{-1-\varepsilon}$ where $C$ depends only on $\varepsilon$ and $p$
\cite[Thm 5.5]{duren}. Thus, using Fubini's theorem to evaluate the 
right hand side of \eqref{inegfracdb}, we deduce that 
$H^p$ embeds in $W^{\theta,p}(\DD)$ for $\theta\in(0,1/p)$ and 
we get by the Sobolev embedding theorem that
\begin{equation}
\label{sum2p}
\|f\|_{L^\lambda(\D)}\leq C\|f\|_{H^p},\qquad p\leq \lambda<2p,
\end{equation}
where $C=C(p,\lambda)$.
When $p=2$ these estimates can be sharpened, for
in this case Green's formula yields that 
$\|(1-|z|^2)^{1/2}\nabla f(z)\|_{L^2(\DD)}\sim
\|f-f(0)\|_{H^2}$
\cite[Ch. VI, Lem. 3.2]{garnett}, hence it follows from 
\eqref{inegfracdb} that $H^2\subset W^{1/2,2}(\DD)$ and subsequently, by the
Sobolev embedding theorem, that $\|f\|_{L^4(\DD)}\leq C\|f\|_{H^2}$.
In fact, since both sides of \eqref{inegfracdb} are equivalent 
quantities when $f$ is harmonic (see discussion after \eqref{inegfracdb}),
$H^2$ is precisely the
space of holomorphic functions in $W^{1/2,2}(\DD)$ with equivalence of 
{ norms. A nonzero} $f\in H^p$ is such that $\log|f_{|\TT}|\in L^1(\TT)$
\cite[Thm 2.2]{duren}, in particular, 
\emph{a nonzero $H^p$-function cannot vanish 
on a subset  of $\TT$ of strictly positive measure}\footnote{More generally, it is a theorem of Privalov that no 
nonzero meromorphic function on $\DD$ has 
nontangential limit zero on a set of strictly positive measure on $\TT$,
see \cite[Sec. 6.1]{Pommerenke}.}.
Conversely,
if $h\in L^p(\TT)$ is non-negative and $\log h\in L^1(\TT)$,
then
\begin{equation}
\label{defouter}
E_{h}(z)=\exp\left\{\frac{1}{2\pi}\int_\TT \frac{\xi+z}{\xi-z}\log h(\xi)\,dm(\xi)
\right\},
\qquad
z\in\DD,
\end{equation}
belongs to $H^p$ and satisfies $|(E_h)_{|\TT}|=h$. 
A function of the form \eqref{defouter}
is called \emph{outer}, and it is characterized among $H^p$ functions by the fact that $\log|E_h|$ (which is harmonic in $\DD$ since $E_h$
has no zeros there) is the
Poisson integral of its nontangential limit. Each nonzero $f\in H^p$ can be 
factored as
$f=J E_{|f|}$ where $J$ is \emph{inner}, meaning that  $J\in H^\infty$ and
$|J_{|\TT}|\equiv1$
\cite[Thm 2.8]{duren} \cite[Ch. II, Cor. 5.7]{garnett}. Conversely, every
product  $J E_h$ where $J$ is inner and $h$ as in \eqref{defouter}  
is a member of $H^p$. 
The multiplicative decomposition 
$f=JE_{|f|}$ is called the inner-outer factorization of $f$.
We shall need that if $f\in H^p$, $g\in H^q$, and $|f_{|\TT}g_{|\TT}|\in 
L^r$ for some $p,q,r\geq1$, then $fg\in H^r$. Indeed,
one has inner-outer factorizations $f=J_1 E_{|f|}$ and $g=J_2 E_{|g|}$,
so that $fg=J_1J_2 E_{|f|} E_{|g|}=J_1J_2 E_{|fg|}$; now,
$J_1J_2$ is inner and $E_{|fg|}\in H^r$ since
$\log|f_{|\TT}g_{|\TT}|=\log|f_{|\TT}|+\log|g_{|\TT}|\in L^1(\TT)$ 
and $|fg|_{|\TT}\in L^r(\TT)$, whence $fg\in H^r$.

Every real-valued
harmonic function $u$ on $\DD$ has a {\it harmonic conjugate},
that is, a real-valued harmonic function $v$
on $\DD$ such that $u+iv$ is holomorphic;
this follows by simple connectedness of $\DD$  from the fact that
$\Delta u=0$ makes $-\partial_{x_2}udx_1+\partial_{x_1}udx_2$ an exact 
differential. 
The conjugate function 
is defined up to an additive constant, and we customarily normalize it so that
$v(0)=0$. 
When $1\leq p\leq\infty$ and real $\psi\in L^p(\TT)$, then $u(z)=P[\psi](z)$
is harmonic on $\DD$ and it is a theorem of Fatou that it has
nontangential limit $\psi$ a.e. on $\TT$. Also, it holds that
$\|u\|_{L^p(\TT_\rho)}\leq \|\psi\|_{L^p(\TT)}$ for $0\leq\rho<1$.
Under the stronger assumption that $1<p<\infty$, 
then $v=P[\widetilde{\psi}]$ where
\begin{equation}
\widetilde{\psi}(e^{i\theta}) := \lim_{\varepsilon\to0}
\frac{1}{2 \, \pi} \, 
\int_{\varepsilon<|\theta-t|<\pi} \frac{\psi(e^{it})}{\tan(\frac{\theta-t}{2})}
\, dm(t)
\label{RHt}
\end{equation}
is called the \emph{conjugate function} of $\psi$. 
It is a theorem of
M. Riesz that the conjugation operator $\psi\mapsto\widetilde{\psi}$
is an isomorphism of $L^p(\T)$ when $1<p<\infty$. 
Thus, we see that if $\psi\in L^p_\RR(\TT)$
and $1<p<\infty$, then there exists $g\in H^p$ (namely $g=u+iv$)
such that $\mbox{Re}\,g=\psi$ on $\T$ \cite[Ch. III]{garnett}. 
Such a $g$ is
unique up to addition  of a pure imaginary constant,
and if we normalize it so that $\text{Im}\,g(0)=0$, then
$\|g\|_{H^p}\le C 
\|\psi\|_{L^p(\T)}$ with  $C=C(p)$.

When $\psi\in L^1(\T)$, the conjugate function $\widetilde{\psi}$ is still
defined pointwise almost everywhere {\it via} \eqref{RHt} but it may
no longer belong to $L^1(\T)$.

For $p\in(1,\infty)$, a non-negative function $\mathfrak{w}\in L^1(\T)$ 
is said to satisfy Muckenhoupt condition $A_p$ if 
\begin{equation}
\label{defAp}
\{\mathfrak{w}\}_{A_p}:=\sup_{I}\Bigl(\frac{1}{m(I)}\int_{I} \mathfrak{w}\,dm\Bigr)
\Bigl(\frac{1}{m(I)}\int_I \mathfrak{w}^{-1/(p-1)} dm\Bigr)^{p-1}<+\infty,
\end{equation}
where the supremum is taken over all arcs $I\subset\T$.
A theorem of Hunt, Muckenhoupt and Wheeden 
\cite[Ch. VI, Thm 6.2]{garnett}\footnote{The proof given there on 
the half-plane easily carries over to the disk.}  
asserts that $\mathfrak{w}$ satisfies condition $A_p$ if and only if there is
$C>0$ independent of $\phi$ for which
\begin{equation}
\label{HMW}
\int_\T |\widetilde{\phi}|^p\,\mathfrak{w}\,dm
\leq C \int_\T |\phi|^p\,\mathfrak{w}\,dm,
\qquad \phi\in L^1(\T),
\end{equation}
and also that \eqref{HMW} is equivalent to 
\begin{equation}
\label{MMW}
\int_\T |M\phi|^p\,\mathfrak{w}\,dm
\leq C_1 \int_\T |\phi|^p\,\mathfrak{w}\,dm
\end{equation}
for $C_1 > 0$ and where $M\phi$ is the Hardy-Littlewood maximal function of $\phi$:
\begin{equation}
\label{defHLM}
M\phi(\xi)=\sup_{I\ni \xi}\frac{1}{m(I)}\int_I|\phi|\,dm,\qquad \xi\in\TT,
\end{equation}
the supremum being taken over all subarcs of $\TT$ that contain $\xi$.
In \eqref{HMW},  the 
assumption $\phi\in L^1(\T)$ is just a means to ensure that $\widetilde{\phi}$
is well defined $m$-a.e. 
and the constants $C$, $C_1$ 
can be chosen to depend only on 
$\{\mathfrak{w}\}_{A_p}$. 

Condition $A_2$ is
fundamental to function theory on Lipschitz (and more generally
chord-arc\footnote{A Jordan domain $\Omega$  is chord-arc
(or Lavrentiev) if
$\Lambda(J(\xi_1,\xi_2))\leq M |\xi_1-\xi_2|$ whenever 
$\xi_1,\xi_2\in\partial\Omega$, where $J(\xi_1,\xi_2)$ is the smaller arc of 
$\partial\Omega$ between $\xi_1$ and $\xi_2$ and $M$ is a constant.}) domains,
as was
first pointed out in the seminal work \cite{Kenig}, see also
\cite{JKAinf,wolff,Zinsmeister}. 
Recall from the Riemann mapping theorem that to
each simply connected domain $\Omega\subset\CC$ there is
a conformal map $\varphi$ from $\DD$ onto $\Omega$,
which is unique if we impose for instance 
$\varphi(0)\in \Omega$ and $\arg\varphi'(0)\in[0,2\pi)$. 
The precise normalization is unimportant in what follows.
\begin{lemma}
\label{conforL}
Let $\Omega\subset\CC$ be a bounded simply connected Lipschitz domain
and $\varphi:\DD\to\Omega$ a conformal map. Then $\varphi$ 
extends homeomorphically from $\overline{\D}$ onto $\overline{\Omega}$
and preserves nontangential regions of approach in that, to every 
$\alpha,\beta>1$, 
there are $\alpha',\beta'>1$ such that:
\begin{equation}
\label{prenta}
R^\Omega_{\alpha}(\varphi(\xi))\subset
\varphi\left(R^\DD_{\alpha'}(\xi)\right)\quad \mbox{\rm and}\quad
\varphi\left(R^\DD_{\beta}(\xi)\right)
\subset R^\Omega_{\beta'}(\varphi(\xi)), \qquad\xi\in\TT.
\end{equation}
The derivative $\varphi'$ as well as its reciprocal $1/\varphi'$
lie in $H^p$ for some $p>1$, and it holds
for any measurable 
$E\subset\partial\Omega$ that $\Lambda(E)=\int_{\varphi^{-1}(E)}|\varphi'|dm$.
Moreover, the weights
$|\varphi_{|\TT}'|$ and $1/|\varphi_{|\TT}'|$ satisfy condition $A_2$.
\end{lemma}
\begin{proof}
Since $\partial\Omega$ is a Jordan curve,
$\varphi$ extends to a homeomorphism from $\overline{\DD}$ onto
$\overline{\Omega}$ mapping $\T$ to $\partial\Omega$
by Carath\'eodory's theorem
\cite[Thm 2.6]{Pommerenke}. 
To prove \eqref{prenta}, we follow the argument (attributed to F. Gehring)
outlined in \cite[Prop. 1.1]{JKAinf}
for conformal maps from a half-plane onto 
unbounded chord-arc { domains. 
Observe} first that $\Omega$
is {\it a fortiori} chord-arc since it is  Lipschitz,
in particular $\Omega$ is a quasi-disk\footnote{Same definition as
a chord-arc domain except that $\Lambda$ gets  replaced by ``diameter''.}
 \cite[Prop. 7.7]{Pommerenke},
and consequently $\varphi$ 
extends to a quasi-conformal homeomorphism\footnote{An orientation-preserving
 homeomorphism 
$\varphi\in W^{1,2}_{loc}(\CC)$  is quasi-conformal if $\|\bar\partial \varphi/\partial \varphi\|_{L^\infty(\CC)}<1$, see \cite[Def. 2.5.2, Thm 2.5.4]{AIM}. }of $\CC$
\cite[Thm 5.17]{Pommerenke}. Such a map is quasi-symmetric
\cite[Def. 3.2.1, Thm3.5.3]{AIM},
meaning that there is an increasing homeomorphism $\eta$ of
$[0,\infty)$ such that
\begin{equation}
\label{QS}
\left|\frac{\varphi(z_0)-\varphi(z_1)}{\varphi(z_0)-\varphi(z_2)}\right|
\leq \eta\left(\left|\frac{z_0-z_1}{z_0-z_2}\right|\right),\qquad
z_0,z_1,z_2\in\CC.
\end{equation}
Now, fix  $\beta>1$, $z_1\in\TT$ and let $z_0$ range over $R^\DD_\beta(z_1)$.
If we choose $z_2\in\TT$ such that 
$|\varphi(z_0)-\varphi(z_2)|=\textrm{d}(\varphi(z_0), \partial\Omega)$, 
then it follows from \eqref{QS} that
\begin{equation}
\label{inegQS}
\frac{|\varphi(z_0)-\varphi(z_1)|}
{\textrm{d}(\varphi(z_0),\partial\Omega)}
\leq \eta\left(\left|\frac{z_0-z_1}{z_0-z_2}\right|\right)
\leq \eta\left(\frac{|z_0-z_1|}{\textrm{d}(z_0,\TT)}\right)
\leq \eta(\beta),
\end{equation}
hence $\varphi(z_0)\in R^\Omega_{\beta'}(\varphi(z_1))$ with
$\beta'=\eta(\beta)$.
This proves the second inclusion in \eqref{prenta} and the first follows 
in the same manner, replacing $\varphi$ by its inverse which is also
quasi-conformal \cite[Thm 3.7.7]{AIM}.

Next, since $\partial\Omega$ is rectifiable, $\varphi'$ 
lies in $H^1$ and 
$\Lambda(E)=\int_{\varphi^{-1}(E)}|\varphi'|dm$
for every measurable $E\subset\partial\Omega$  \cite[Thm 6.8]{Pommerenke}.
 The fact that $|\varphi'_{|\TT}|$ 
meets $A_2$ is a consequence of \cite[Prop. 15]{wolff}
(which deals more generally with local chord arc graphs),
see also 
\cite[Sec. 2]{lanzani-stein} and the references therein 
or \cite[Ch. VII, Thm 4.2]{GM} for a proof 
when $\Omega$ is star-shaped.
The fact that $|\varphi'_{|\TT}|$ satisfies
condition $A_2$ implies
that it belongs to $L^{1+\delta}(\TT)$ for some $\delta>0$
\cite[Ch. VI, Cor. 6.10]{garnett}, hence 
it holds in fact that $\varphi'\in H^p$ for some $p>1$.
Clearly
$\{|\varphi'_{|\TT}|\}_{A_2}=\{1/|\varphi'_{|\TT}|\}_{A_2}$.
As $\Omega$ is  chord-arc, it is in particular a Smirnov
domain,
meaning that $\varphi'$ is outer \cite[Sec. 7.3, 7.4]{Pommerenke}. 
Hence $1/\varphi'$ is also outer and since 
$1/|\varphi'_{|\TT}|\in L^{1+\delta}(\TT)$ for some $\delta>0$ because
it satisfies $A_2$, we find that in turn
$1/\varphi'\in H^p$ for some $p>1$.
\end{proof}

\subsection{Smirnov classes of a Lipschitz plane domain}
\label{sec:Smirnov}

On an arbitrary simply connected domain $\Omega$ 
(whose boundary contains more than one point), 
there are at least two generalizations 
of the Hardy space $H^p$ of the disk. One which
goes by the name of Hardy space,
but is of no concern to us here,
requires $|f|^p$ to have a harmonic majorant on $\Omega$. 
The other, which is the one we are interested in,
is the so-called Smirnov space, denoted as $\mathcal{S}^p(\Omega)$.
It consists of functions 
$f$, holomorphic in $\Omega$, for which there is a sequence of 
relatively compact Jordan domains $\Delta_n\subset\Omega$
 with rectifiable boundary such that
each compact $K\subset\Omega$ is contained in $\Delta_n$ for $n\geq n(K)$ and
\begin{equation}
\label{systema}
\sup_{n\in\NN}\|f\|_{L^p(\partial\Delta_n)}<\infty.
\end{equation}
By the maximum principle 
$\mathcal{S}^\infty(\Omega)$ consists of bounded holomorphic functions on 
$\Omega$. When $1\leq p<\infty$ it is not immediately clear 
that $\mathcal{S}^p(\Omega)$ is a Banach space, but this is nevertheless true 
and there is in fact a fixed sequence $\Delta_n$
such that (\ref{systema}) holds for all $f\in\mathcal{S}^p(\Omega)$. 
Such a sequence can
be taken to be $\varphi(\DD_{\rho_n})$ where  $\rho_n\to1^-$ and
$\varphi$ is a conformal map from $\DD$ onto $\Omega$
\cite[Thm 10.1]{duren}. Consequently $f$ belongs to $\mathcal{S}^p(\Omega)$
if and only if $(f\circ\varphi)(\varphi')^{1/p}$ belongs to $H^p$, and
$\|(f\circ\varphi)(\varphi')^{1/p}\|_p$ will serve as a norm on 
$\mathcal{S}^p(\Omega)$ \cite[Ch. 10, Sec. 1, Cor. to Thm 10.1]{duren}. 

As soon as $\partial\Omega$ is rectifiable, so that $\varphi'\in H^1$,
the previous characterization together with Lemma \ref{conforL}
and the discussion in Section \ref{sec:Hardy} imply that  
each $f\in\mathcal{S}^p(\Omega)$ has nontangential limits a.e. on
$\partial\Omega$ with respect to arclength, and that
the boundary function
thus defined lies in $L^p(\partial\Omega)$. Moreover, \emph{this boundary function
cannot vanish on a set of positive
arclength unless $f\equiv0$}, and its norm in $L^p(\partial\Omega)$ coincides
with $\|f\|_{\mathcal{S}^p(\Omega)}$, 
thereby identifying $\mathcal{S}^p(\Omega)$ with a closed 
subspace of $L^p(\partial\Omega)$. 
Again $f$ is recovered from its boundary 
function by a Cauchy integral \cite[Thm 10.4]{duren}, but
the (analog of the) Poisson representation may now fail.

Our interest in Smirnov spaces is here limited to
$\mathcal{S}^2(\Omega)$ for $\Omega$ 
a bounded simply connected Lipschitz domain.  Theorem \ref{equivS2}
below gives two alternative descriptions of this space.
{ We mention that the analog of point $(i)$ 
for unbounded chord-arc domains 
is contained in \cite[Thm 2.2]{JKAinf}.}
First, we need a lemma:
\begin{lemma}
\label{S2H1}
Let $\Omega\subset\CC$ be a bounded simply connected Lipschitz domain,
and $\varphi$ map $\DD$ conformally onto $\Omega$.
If $f\in\mathcal{S}^2(\Omega)$, then $f\circ\varphi\in H^1$.
\end{lemma}
\begin{proof}
Set for simplicity  $F=f\circ\varphi$.
Since $f\in \mathcal{S}^2(\Omega)$, we have that $F(\varphi')^{1/2}\in H^2$,
and we know from Lemma \ref{conforL} that $1/\varphi'$ lies
in $H^1$. Therefore, by the Schwarz inequality and the monotonicity
of $\rho\mapsto\|g_\rho \|_{L^p(\T)}$  for holomorphic $g$,
we get that
\[\left(\int_{\TT_\rho}|F|dm\right)^2\leq
\int_{\TT_\rho}|F|^2|\varphi'|dm
\int_{\TT_\rho}|1/\varphi'|dm\leq
\left\|F(\varphi')^{1/2}\right\|^2_{H^2}\|1/\varphi'\|_{H^1}.
\] 
\end{proof}

\begin{theorem}
\label{equivS2}
Let $\Omega\subset\CC$ be a bounded simply connected Lipschitz domain.
\begin{itemize}
\item[(i)] For each $\alpha>1$, the space $\mathcal{S}^2(\Omega)$ 
coincides with 
holomorphic functions $f$ in $\Omega$ such that 
$\mathcal{M}_\alpha f\in L^2(\partial\Omega)$ and
$f\mapsto\|\mathcal{M}_\alpha f\|_{L^2(\partial\Omega)}$ is an 
equivalent norm on
$\mathcal{S}^2(\Omega)$.
\item[(ii)] $\mathcal{S}^2(\Omega)$ is the closed subspace of 
$W^{1/2,2}(\Omega)$ consisting of holomorphic functions,
with equivalence of norms.
\end{itemize}
\end{theorem}
\begin{proof}
Let $\varphi$ 
map $\DD$ conformally onto $\Omega${, put $\psi$ for the inverse map,} and
pick $\alpha>1$.  By \eqref{prenta}, there is $\beta>1$
such that $\mathcal{M}_\alpha f\leq  
(\mathcal{M}_\beta F)\circ{ \psi}$ for $f\in \mathcal{S}^2(\Omega)$
and $F=f\circ\varphi$. From  Lemma \ref{S2H1} we get that $F\in H^1$,
hence it is the Poisson integral of $F_{|\TT}$.
It is  known, however, that 
$\mathcal{M}_\beta F\leq C M F_{|\TT}$ pointwise on
$\TT$ for some constant $C$ depending only on $\beta$ 
\cite[Ch. I, Thm 4.2]{garnett}\footnote{The proof given there
on the half-plane carries over immediately to the disk.}.
Consequently,
\begin{equation}
\label{majMF1}
\begin{array}{rl}
\int_{\partial\Omega} \left(\mathcal{M}_\alpha f\right)^2d\Lambda&\leq
\int_{\partial\Omega} \left(\mathcal{M}_\beta F\circ{ \psi}\right)^2
d\Lambda\\
&=
\int_{\TT} \left(\mathcal{M}_\beta F\right)^2|\varphi'|dm
\leq
C^2\int_\TT (MF)^2|\varphi'|dm,
\end{array}
\end{equation}
where the change of variable is justified by Lemma \ref{conforL}.
Now, as $|\varphi'|$ satisfies condition $A_2$,
we get in view of \eqref{MMW} that
\begin{equation}
\label{majMF2}
\int_\TT (MF)^2|\varphi'|\,dm\leq C_1^2 \int_\TT |F|^2|\varphi'|\,dm
=C_1^2\|f\|^2_{\mathcal{S}^2(\Omega)}
\end{equation}
for some $C_1$ depending only on $\{|\varphi'|\}_{A_2}$.
From \eqref{majMF1} and \eqref{majMF2}, it follows that
\begin{equation}
\label{MaxSmirn}
\|\mathcal{M}_\alpha f\|_{L^2(\partial\Omega)}\leq C_2\|f\|_{\mathcal{S}^2(\Omega)}
\end{equation}
with $C_2=C_2(\Omega, \alpha)$.
Conversely,
assume  that $f$ is holomorphic in $\Omega$ with 
$\|\mathcal{M}_\alpha f\|_{L^2(\partial\Omega)}<\infty$. 
{ Whenever $\delta\in(1,\infty)$ and $z_0\in\Omega$,
it is a famous estimate for harmonic functions on 
Lipschitz domains (in any dimension)
that
\begin{equation}
\label{estDahl}
\|\mathcal{M}_\delta (f-f(z_0))\|_{L^2(\partial\Omega)}\sim
\|\textrm{d}(.,\partial\Omega)^{1/2}\nabla f\|_{L^2(\Omega)}
\end{equation}
where  the constants depend only on $\Omega$, $\delta$ and $z_0$
\cite[Thm 1, Cor. 1]{DahlbergNT}. Assume first that $f(\varphi(0))=0$,
in which case it follows from 
\eqref{estDahl} that
\begin{equation}
\label{CompNT}
\|\mathcal{M}_\delta f\|_{L^2(\partial\Omega)}\leq C
\|\mathcal{M}_\alpha f\|_{L^2(\partial\Omega)}<\infty,
\end{equation}
where 
$C=C(\alpha,\delta,\Omega)$. Pick $\delta=\eta(2)$, where $\eta$ is as in
\eqref{QS}; note that indeed $\eta(2)>1$,  since $\eta$ is strictly 
increasing and $\eta(1)\geq1$.
Now, the argument in \eqref{inegQS} can be reversed (set $\beta=2$ there)
so that,
if $z_1\in\TT$, then $R_\delta^{\Omega}(\varphi(z_1))\supset 
\varphi(R^\DD_2(z_1))$. Hence
$\mathcal{M}_{2}F\leq(\mathcal{M}_\delta f)\circ\varphi$,
$F=f\circ\varphi$, and for $\rho\in[0,1)$ we get from
the obvious inequality
$|F(\rho e^{i\theta})|\leq\mathcal{M}_2 F(e^{i\theta})$ that}
\begin{equation}
\nonumber
\int_{{ \TT}} |F_{ \rho}|^2|\varphi'|\,dm
\leq\int_\TT
|\mathcal{M}_{{ 2}}F|^2|\varphi'|\,dm
\leq
\int_\TT|(\mathcal{M}_{{ \delta}}f) \circ\varphi|^2|\varphi'|\,dm
=\left\|\mathcal{M}_{{ \delta}}f\right\|^2_{L^2(\partial\Omega)}\, .
\end{equation}
{ In view of \eqref{CompNT}, the previous inequality shows
that $\|F_\rho(\varphi')^{1/2}\|_{H^2}$
is bounded independently of $\rho$, so there is a sequence 
$\rho_k\to1^-$ such that  $F_{\rho_k}(\varphi')^{1/2}$ converges weakly in
$H^2$ to some function $G$. Since $F_{\rho_k}(z)=F(\rho_kz)$ 
converges to $F(z)$ locally 
uniformly in $\DD$, passing to the weak limit in the Cauchy formula
yields $G=F(\varphi')^{1/2}$. As the norm of the weak limit cannot exceed the
$\liminf$ of the norms, we deduce on using \eqref{CompNT}  that
\begin{equation}
\label{inegunw}
\|f\|_{\mathcal{S}^2(\Omega)}=
\|F(\varphi')^{1/2}\|_{H^2}\leq
\|\mathcal{M}_\delta f\|_{L^2(\partial\Omega)}\leq
C\|\mathcal{M}_\alpha f\|_{L^2(\partial\Omega)}. 
\end{equation}
Finally, if $f(\varphi(0))\neq0$, we apply \eqref{inegunw} 
to $\psi f$ which has the same $\mathcal{S}^2(\Omega)$-norm as $f$ and
does vanish at $\varphi(0)$ (that
$\|\psi f\|_{\mathcal{S}^2(\Omega)}=\|f\|_{\mathcal{S}^2(\Omega)}$ is clear
from the relation  $(\psi f)\circ\varphi(z)=z f(\varphi(z))$).
Since $\mathcal{M}_\alpha \psi f\leq \mathcal{M}_\alpha f$ because 
$|\psi|\leq1$, we get that
$\|f\|_{\mathcal{S}^2(\Omega)}\leq
C\|\mathcal{M}_\alpha f\|_{L^2(\partial\Omega)}$ where $C=C(\alpha,\Omega)$,}
thereby proving $(i)$.

As for
$(ii)$,
since holomorphic functions are harmonic, we know 
from \cite[Thm 4.1]{JKinhom} that 
\begin{equation}
\label{JKeq}
\|f\|_{L^2(\DD)}+\|\textrm{d}(.,\partial\Omega)^{1/2}\nabla f\|_{L^2(\Omega)}\sim \|f\|_{W^{1/2,2}(\Omega)},\qquad f\in\mathcal{S}^2(\Omega),
\end{equation}
where the constants depend only on $\Omega$.
{{Pick $z_0 \in \Omega$}}. As $f(z_0)$ is the mean of $f$ over some disk $\DD_{z_0,\rho_0}\subset\Omega$,
we get that
$|f(z_0)|\leq C_3\|f\|_{L^2(\Omega)}$ where $C_3=C_3(z_0,\Omega)$.
From this, together with \eqref{JKeq} and \eqref{estDahl}, it follows that
\begin{align}
\label{S2W1/2}
\|\mathcal{M}_\alpha f\|_{L^2(\partial\Omega)}&\leq
|f(z_0)|+\|\mathcal{M}_\alpha (f-f(z_0))\|_{L^2(\partial\Omega)}\\
\nonumber
&\leq C_4\bigl(\|f\|_{L^2(\DD)}+\|\textrm{d}(.,\partial\Omega)^{1/2}\nabla f\|_{L^2(\Omega)}\bigr)\sim \|f\|_{W^{1/2,2}(\Omega)}.
\end{align}
Conversely, the Schwarz inequality implies that
\[\|f\|_{L^2(\Omega)}^2=\int_\DD |(f\circ\varphi)|^2|\varphi'|^2dm_2
\leq \|(f\circ\varphi)(\varphi')^{1/2}\|_{L^4(\Omega)}^2
\|\varphi'\|_{L^2(\Omega)}^2,
\]
{ and since $H^2$ embeds in $L^4(\DD)$ (see discussion after \eqref{sum2p})
while $(\varphi')^{1/2}\in H^2$ by Lemma \ref{conforL},
we get that}
$\|f\|_{L^2(\Omega)}\leq C_5\|f\|_{\mathcal{S}^2(\Omega)}$ where
$C_5=C_5(\Omega,\varphi)$.
From this together with \eqref{JKeq}, \eqref{estDahl}, and the inequality 
$|f(z_0)|\leq C_3\|f\|_{L^2(\Omega)}$ already mentioned, we obtain:
\begin{align}
\label{W1/2S2}
\|f\|_{W^{1{ /}2,2}(\Omega)}&\sim 
\|f\|_{L^2(\Omega)}+\|\textrm{d}(.,\partial\Omega)^{1/2}\nabla f\|_{L^2(\Omega)}\\
\nonumber
&\leq C_6\bigl(\|f\|_{\mathcal{S}^2(\Omega)}+\|\mathcal{M}_\alpha f\|_{L^2(\partial\Omega)}+|f(z_0)|\bigr)\\
\nonumber
&\leq C_7\bigl(\|f\|_{\mathcal{S}^2(\Omega)}+\|\mathcal{M}_\alpha f\|_{L^2(\partial\Omega)}\bigr).
\end{align}
Now, point $(ii)$ follows from \eqref{S2W1/2}, \eqref{W1/2S2} and point
$(i)$.
\end{proof}

\subsection{Smirnov spaces and { Dirichlet problems for the Laplacian}}
\label{Smha}

{ Let $\mathfrak{t}\in L^\infty(\partial\Omega)$ be the tangent vector field to
$\partial\Omega$ written in complex form: 
$\mathfrak{t}=\tau_{x_1}+i\tau_{x_2}$  $\Lambda$-a.e. in $\partial\Omega$.} 
\begin{proposition}
\label{Hdtau}
Let $\Omega\subset\CC$ be a bounded simply connected Lipschitz domain
and $H\in W_\RR^{3/2,2}(\Omega)$ be a harmonic function. Then 
$\partial H\in\mathcal{S}^2(\Omega)$ and
\begin{equation}
\label{exptd}
\partial_\tau H=2\textrm{Re}\left(\partial H \,\mathfrak{t}\right) \, .
\end{equation}
In particular, 
$\tr_{\partial\Omega} H\in W_\RR^{1,2}(\partial\Omega)$.
\end{proposition}
\begin{proof}
Since $H$ is harmonic, $\partial H$ is holomorphic,
and  $\partial H\in W^{1/2,2}(\Omega)$ because
$H\in W_\RR^{3/2,2}(\Omega)$. Thus, 
$\partial H\in\mathcal{S}^2(\Omega)$ by Theorem \ref{equivS2} $(ii)$.
Also, by the Sobolev embedding theorem, $H$ is continuous on 
$\overline{\Omega}$. Let $\varphi$ map $\DD$ conformally onto $\Omega$.
Lemma \ref{conforL} implies that $u:=H\circ\varphi$ is harmonic on $\DD$ 
and continuous on $\overline{\DD}$. Moreover, the complex chain rule
\cite[Ch. 1, Sec. C]{ahlfors} gives us, since $\bar\partial \varphi=0$,
that
\begin{align}
\nonumber
d u&=(\partial H\circ\varphi)\,\varphi'\,dz+
(\bar\partial H\circ\varphi)\,\overline{\varphi'}\,d\bar z\\
\label{diffcomp}
&=2\text{\rm Re}\Bigl((\partial H\circ\varphi)\,\varphi'\,dz\Bigr),
\end{align}
where we used that $\bar \partial H=\overline{\partial H}$. Now,
to say that 
$\partial H\in\mathcal{S}^2(\Omega)$ is equivalent to say that
$(\partial H\circ\varphi)(\varphi')^{1/2}\in H^2$, and 
Lemma \ref{conforL} implies that $(\varphi')^{1/2}\in H^2$, hence
$F:=(\partial H\circ\varphi)\varphi'\in H^1$. In particular, $F_\rho$
converges to $F_{|\TT}$ in $L^1(\TT)$ as $\rho\to1^-$ and therefore,
by integration, we get from \eqref{diffcomp} upon setting
$\varphi(e^{i\theta_j})=\zeta_j\in\partial\Omega$, $j=1,2$, that
\begin{align}
\nonumber
H(\zeta_1)-H(\zeta_2)&=\lim_{\rho\to1^-}(u(\rho e^{i\theta_1})-
u(\rho e^{i\theta_2}))=\lim_{\rho\to1^-}2
\int_{\theta_1}^{\theta_2}\textrm{Re}\left(F(\rho e^{i\theta})ie^{i\theta}
\right)
\rho d\theta\\
\label{intbT}
&=2\int_{\theta_1}^{\theta_2}\textrm{Re}\left((\partial H\circ\varphi)(e^{i\theta})
\varphi'(e^{i\theta})ie^{i\theta}\right)
d\theta.
\end{align}
Since $\varphi'(e^{i\theta})ie^{i\theta}/|\varphi'(e^{i\theta})|=\mathfrak{t}(\varphi(e^{i\theta}))$ and $d\Lambda=|\varphi'(e^{i\theta})|d\theta$ by Lemma \ref{conforL}, we may rewrite \eqref{intbT} as
\[H(\zeta_1)-H(\zeta_2)=2\int_{[\zeta_1,\zeta_2]}\textrm{Re}\left(\partial H
\mathfrak{t}\right)d\Lambda,\]
where $[\zeta_1,\zeta_2]$ is the oriented arc from $\zeta_1$ to $\zeta_2$ on 
$\partial\Omega$.
This proves \eqref{exptd}.
\end{proof}
\begin{remark}
That $\textrm{tr}_{\partial\Omega} H$ belongs to
$W^{1,2}_\RR(\partial\Omega)$ in
{ Proposition \ref{Hdtau}} depends on the fact that $H$ is harmonic, and
is \emph{not}
a general property of $W^{3/2,2}(\Omega)$-functions, see the discussion 
before \cite[Prop. 3.2]{JKinhom} 
for a counterexample credited to G. David.
\end{remark}
In view of Theorem \ref{equivS2}, the next proposition 
stands analog in the planar case to a well-known result 
on the Dirichlet problem obtained in \cite[Thm 5.1]{Verchota}
for $n\geq3$. The proof we give here in the planar case is quite
different, and uses conformal mapping
and the M. Riesz theorem as global tools\footnote{The
restriction to $n\geq3$ in \cite{Verchota} may be due to the fact that 
it dwells on the method of layer potentials, { where}
the discrepancy between
Riesz and logarithmic potentials makes it cumbersome to treat both
in a single stroke.}.

\begin{proposition}
\label{Neucor}
Let $\Omega\subset\CC$ be a bounded simply connected Lipschitz domain,
and $\psi\in L^2_\RR(\partial\Omega)$ be such that 
$\int_{\partial\Omega}\psi d\Lambda=0$. Then, there is a harmonic
function $U\in W^{3/2,2}_{\RR}(\Omega)$  such that 
$\partial_\tau\text{\rm tr}_{\partial\Omega}U=\psi$. Such a function
is unique up to an  additive real constant and
$\partial U\in\mathcal{S}^2(\Omega)$ with 
$\|\partial U\|_{\mathcal{S}^2(\Omega)}\leq C\|\psi\|_{L^2(\partial\Omega)}$, 
where $C$ depends only on $\Omega$. 
\end{proposition}
\begin{proof}
As  $U$ is continuous on $\overline{\Omega}$
by the Sobolev embedding theorem, and 
$\tr_{\partial\Omega}U\in W^{1,2}(\partial\Omega)$
by { Proposition \ref{Hdtau}},
uniqueness follows from the maximum principle for harmonic functions.

Next, let $\varphi$ map $\DD$ conformally onto $\Omega$ and {{$\Upsilon$ be the inverse map. Define
$h=(\psi\circ\varphi)|\varphi'|$ on $\TT$. By Lemma \ref{conforL}
$\|\psi\|_{ L^2(\partial\Omega)}=\|h(\varphi')_{|\TT}^{-1/2}\|_{L^2(\TT)}$
and 
$(\varphi')_{|\TT}^{1/2}\in L^\ell(\TT)$ for some $\ell>2$.
Therefore  $h\in L^p(\TT)$ for some $p>1$, by H\"older's inequality.}}
Moreover $\int_\TT h dm=\int_{\partial\Omega}\psi d\Lambda=0$, hence
by the M. Riesz theorem there 
is $G\in H^p$ such that $G(0)=0$ and $\text{\rm Re} \, G_{|\TT}=h$.
Because $\|\psi\|_{ L^2(\partial\Omega)}=
\|h(\varphi')_{|\TT}^{-1/2}\|_{L^2(\TT)}$
and $1/|\varphi'|$ meets  condition
$A_2$ by Lemma \ref{conforL}, we get from \eqref{HMW} that 
$\|G_{|\TT}(\varphi')_{|\TT}^{-1/2}\|_{L^2(\TT)}\leq 
C\|\psi\|_{L^2(\partial\Omega)}$ with $C=C(\varphi)$. 
Therefore, as $G\in H^p$ while $(\varphi')^{-1/2}\in H^2$ by Lemma
\ref{conforL}, the product $H=G(\varphi')^{-1/2}$ lies in $H^2$
and $\|H\|_{H^2}\leq C\|\psi\|_{L^2(\partial\Omega)}$.
Since $H(0)=0$ the function $H_1(z)= H(z)/(iz)$ 
in turn lies in $H^2$ with same norm as $H$,
and consequently 
\begin{equation}
\label{appS2tat}
F(\zeta):=\frac{H_1({ \Upsilon}(\zeta))}{\left(\varphi'({ \Upsilon}(\zeta))\right)^{1/2}}
=\frac{G({ \Upsilon}(\zeta))}{i{ \Upsilon}(\zeta)
\varphi'({ \Upsilon}(\zeta))}\in \mathcal{S}^2(\Omega)\
\end{equation}
with $\|F\|_{\mathcal{S}^2(\Omega)}\leq C\|\psi\|_{L^2(\partial\Omega)}$.
In view of Lemma \ref{conforL}, $\zeta\in\Omega$ converges nontangentially to 
$\xi\in\partial\Omega$ if, and only if $z={ \Upsilon}(\zeta)\in\DD$ 
converges nontangentially to $e^{i\theta}={ \Upsilon}(\xi)\in\TT$. 
Since
$ie^{i\theta}\varphi'(e^{i\theta})/|\varphi'(e^{i\theta})|=\mathfrak{t}(\varphi(e^{i\theta}))$,
{ with $\mathfrak{t}$ the tangent vector field in complex form as}
defined before
{ Proposition \ref{Hdtau}, we see from equation
\eqref{appS2tat} and the definition of $G$} that 
\begin{equation}
\label{trF}
\textrm{Re}\left(F(\xi) \,\mathfrak{t}(\xi)
\right)
=\psi(\xi), \qquad \Lambda-\textrm{a.e. } \xi\in\partial\Omega.
\end{equation}

{ Let $U$} be harmonic and real-valued in $\Omega$ with
$\partial U=F/2$. { Clearly $U$ 
exists, for} $F(z)dz+\overline{F(z)}d\bar z$ is a closed 
real-valued differential on the simply connected domain
$\Omega$. { Moreover,} $U\in W_\RR^{3/2,2}(\Omega)$ because 
$\partial U\in W^{1/2,2}(\Omega)$ by Theorem \ref{equivS2} $(ii)$.
Then, it follows from \eqref{trF} and { Proposition \ref{Hdtau}} that
$\partial_\tau U=\psi$. 
\end{proof}

\section{Proof of Theorem \ref{cauchy1}}
\label{sec:proof1}

{
In Section \ref{ssec:fr}, we { state Theorem \ref{normeq3/2}} which is 
instrumental for the proof of  { Theorem \ref{cauchy1}  but is
also} of independent interest. 
{ It}  is proved in Section \ref{sec:62}, along with 
generalizations of results { from Section \ref{Smha} to  more general 
conductivity equations, and a version of Rolle's theorem in $W^{1,2}(\RR)$.
Finally, the proof of  Theorem \ref{cauchy1} is given} in Section 
\ref{subs:proof1}.
}

\subsection{Factorization and regularity}
\label{ssec:fr}
\begin{theorem}
\label{normeq3/2}
Assume that $\Omega\subset\RR^2$ is a bounded Lipschitz 
domain and that $\sigma$ satisfies  \eqref{isot}-\eqref{ellip}.
Let $u \in W^{1,2}_\RR(\Omega)$ be a solution to
\eqref{forcond} { which is
such that} $\partial_n u\in L_\RR^2(\partial\Omega)$. Then:
\begin{itemize}
\item[(i)]$u\in W^{3/2,2}_\RR(\Omega)$ and $\partial u=e^{\Psi}\Phi$ where
$\Psi\in W^{1,r}(\Omega)$ and $\Phi\in \mathcal{S}^2(\Omega)$. Moreover
$\nabla u$ converges nontangentially to 
$\partial_\tau u\,\tau+\partial_n u \,n$ on $\partial\Omega$, and if
$u$ gets normalized so that $u(z_0)=0$ for some
$z_0\in\Omega$, there is a constant 
$C$ depending only on $\Omega$, $z_0$, $r$, $\|\sigma\|_{W^{1,r}(\Omega)}$ 
and $c$ 
in \eqref{ellip}
such that
\begin{equation}
\label{domnu3/2}
\|u\|_{W^{3/2,2}(\Omega)}\leq C\|\partial_n u\|_{L^2(\partial\Omega)}.
\end{equation}
\item[(ii)] For each $\alpha>1$, it holds that 
\begin{equation}
\label{equivmax3/2}
\|\partial_\tau u\|_{L^2(\partial\Omega)}\sim
\|\partial_n u\|_{L^2(\partial\Omega)}\sim
\|\mathcal{M}_\alpha \nabla u\|_{L^2(\partial\Omega)},
\end{equation}
where constants depend only on $\Omega$, $r$,  $\|\sigma\|_{W^{1,r}(\Omega)}$, $c$ 
in \eqref{ellip},  and also on $\alpha$
as to the second equivalence.
\item[(iii)] We have that $u\in W^{2,r}_{\RR,loc}(\Omega)$ and that
\begin{equation}
\label{equivpond3/2}
\sum_{j=1,2}\|\textrm{d}(.,\partial\Omega)^{1/2}\,\partial_{x_j}\nabla u\|^2_{L^2(\Omega)}+
 \|u\|^2_{W^{1,2}(\Omega)} 
\sim\|u\|_{W^{3/2,2}(\Omega)}^2,
\end{equation}
where constants depend only on $\Omega$, $r$, $\|\sigma\|_{W^{1,r}(\Omega)}$, 
{ and} $c$ 
in \eqref{ellip}.
\end{itemize}
\end{theorem}
The proof of Theorem \ref{cauchy1} dwells on the factorization
$\partial u=e^{\Psi}\Phi$ introduced in Theorem \ref{normeq3/2} $(i)$
and on a generalized form of Rolle's theorem given in Proposition \ref{real}.
Roughly speaking, the latter shows that if 
both $u$ and $\partial_n u$ vanish on a subset of 
positive measure of $\partial\Omega$, then 
the full gradient $\nabla u=\partial_\tau u\,\tau+\partial_n u \,n$  also
has to vanish on such a  set.
Consequently, Theorem \ref{normeq3/2} shows that the gradient vanishes
everywhere in $\Omega$, because $\partial u$ factors through a holomorphic 
function of 
Smirnov class which cannot vanish on a subset of positive measure of 
$\partial\Omega$ if it is not identically zero. 

The regularity results needed 
to put this approach to work are set forth in  Theorem \ref{normeq3/2}
points { {$(i)$-$(ii)$. Point $(ii)$ is known,
even in higher 
dimension and with less regular $\sigma$, provided $\Omega$ is starlike
\cite{KP}.  Point $(iii)$ is not used
but mentioned for its own sake, as it generalizes to more 
general conductivities, in the case where $n=2$, the equivalence 
between the two hand sides of \eqref{inegfracdb}
established for harmonic functions in \cite[Thm 4.1]{JKinhom}}}.

\subsection{Proof of Theorem \ref{normeq3/2}}
\label{sec:62}

\subsubsection{The $\sigma$-harmonic conjugate function}
\label{sigmaHC}

When $\Omega\subset\RR^2$ is simply connected,
we observe that
\eqref{forcond} is a compatibility condition for the generalized Cauchy-Riemann
system:
\begin{equation} \label{system}
\left\{
\begin{array}{l}
\partial_{x_1}v=-\sigma\partial_{x_2} u,\\
\partial_{x_2}v=\sigma\partial_{x_1}u,
\end{array}
\right.
\end{equation}
with unknown real-valued functions $u,v$. In fact, 
\eqref{forcond} is equivalent to the Schwarz rule
$\partial_{x_2}\partial_{x_1} v=\partial_{x_1}\partial_{x_2} v$ in 
\eqref{system},
thus there is a distribution $v$ to meet the latter whenever 
$u\in W^{1,2}_\RR(\Omega)$ satisfies
\eqref{forcond} \cite[Ch. II, Sec. 6, Thm VI]{Schwartz}. 
From
\eqref{ellip} and \eqref{system} we get that 
$|\nabla v|\in L_\RR^2(\Omega)$, hence
$v\in W_\RR^{1,2}(\Omega)$ and  
$\partial_\tau\tr_{\partial\Omega}v$ exists in 
$W^{-1/2,2}_\RR(\partial\Omega)$.
The function $v$ is a so-called  $\sigma$-harmonic conjugate to $u$,
and it is unique up to an additive constant by 
\eqref{estfonc}.  

Because \eqref{system} entails that $\nabla v$ is 
the rotation of $\sigma\nabla u$ by $\pi/2$  on $\Omega$,
it may be surmised that 
$\partial_n u=\partial_\tau \tr_{\partial\Omega}v/\sigma\in W_\RR^{-1/2,2}(\partial\Omega)$. This is indeed the case, as follows from
the Green formula 
on Lipschitz domains \cite[Ch. 3, Thm 1.1]{Necas}:
\begin{equation}
\label{GreenNecas}
\int_\Omega(h \, \partial_{x_i} g +g \, \partial_{x_i}h)dm=\int_{\partial\Omega}g \, h \, n_{x_i}\,
d\Lambda \, , \ g,h\in W^{1,2}(\Omega)\, , \  i = 1,2 \, ,  
\end{equation}
where we have put $n=(n_{x_1},n_{x_2})^t$. In fact, since $\tau=(-n_{x_2},n_{x_1})$,
it holds for $\varphi\in\mathcal{D}(\RR^2)$ that 
\[\int_{\partial\Omega}(\partial_\tau v)\varphi\,d\Lambda=
-\int_{\partial\Omega}v\nabla\varphi.\tau\,d\Lambda=
\int_{\partial\Omega}v(n_{x_2}\partial_{x_1}\varphi-n_{x_1}\partial_{x_2}\varphi)\,d\Lambda
\]
so that, by \eqref{GreenNecas},
\begin{equation}
\label{def2nd}
\int_{\partial\Omega}\partial_\tau \, v \, \varphi\,d\Lambda=
\int_\Omega\left(\partial_{x_2}v \, \partial_{x_1}\varphi-\partial_{x_1}v \, \partial_{x_2}\varphi\right)\,dm
=\int_\Omega\sigma \, \nabla u \, . \, \nabla \, \varphi\,dm.
\end{equation}
By density, we conclude on comparing \eqref{defndw} and \eqref{def2nd} that
$\partial_n u=\partial_\tau v/\sigma$, as announced. 
It is easy to check that $v$ satisfies \eqref{forcond} with
$\sigma$ replaced by $1/\sigma$, so the previous discussion also yields that
$\partial_n v=-\sigma \, \partial_\tau u$ on $\partial\Omega$. 
In fact, the peculiarity of the planar case is that solving the 
Neumann problem in $W^{1,2}_\RR(\Omega)$
for the conductivity equation \eqref{forcond}, 
with normal derivative $g\in W^{-1/2,2}_\RR(\partial\Omega)$, is
tantamount to solve  
the Dirichlet problem in $W^{1,2}_\RR(\Omega)$ for a conductivity equation
{ having} conductivity $1/\sigma$ 
{ with} tangential derivative 
$\sigma g\in W^{-1/2,2}_\RR(\partial\Omega)$, 
and then compute the $\sigma$-harmonic conjugate.
In particular 
uniqueness-up-to-a-constant of energy solutions implies that a solution
$u\in W^{1,2}(\Omega)$ to \eqref{forcond} meeting $\partial_\tau u=0$ is 
a constant.

The functions $v$ and $f=u+iv$, which lie respectively
in $W^{1,2}_\RR(\Omega)$ and $W^{1,2}(\Omega)$,
will be instrumental to our analysis. 
For definiteness, we normalize $v$ (initially defined up to an additive 
real constant)  so that $\int_{\partial\Omega}v \, d \Lambda=0$.
A short computation (see \cite[Sec. 3.1]{BLRR})
shows that $f$ satisfies on $\Omega$ the conjugate Beltrami equation:
\begin{equation}
\label{CB}
\bar \partial f=\nu \overline{\partial f},\qquad \nu=(1-\sigma)/(1+\sigma).
\end{equation}
Note that $\|\nu\|_{L^\infty(\Omega)}<1$ and that $\nu\in W^{1,r}_\RR(\Omega)$
because of \eqref{estfonc}, \eqref{isot} and \eqref{ellip}. 
Interior  regularity estimates for \eqref{CB}
imply that $f\in W_{loc}^{2,r}(\Omega)$ \cite[Cor. 3.3]{BFL} 
(see also Section
\ref{facder}),
hence also 
$u,v\in W^{2,r}_{\RR,loc}(\Omega)$.
In particular, by the Sobolev embedding theorem,
$\nabla u,\nabla v$ are locally H\"older 
continuous on $\Omega$.

Let $\{\Omega_k\}$ be
a sequence of open subsets of $\Omega$ with smooth boundary such that 
$\overline{\Omega_k}\subset\Omega_{k+1}$ and $\cup_n\Omega_k=\Omega$.
Whenever $u\in W^{1,2}(\Omega)$ satisfies
\eqref{forcond} and $g\in W^{-1/2,2}(\partial\Omega)$, 
it follows from \eqref{defndw}  (and its analog 
on $\Omega_k$) by means of the Schwarz inequality,
and since $\|\nabla u\|_{L^2(\Omega\setminus\Omega_n)}\to 0$ as $k\to\infty$,
that 
\begin{equation}
\label{seqGreen}
\partial_n u= g\ \Longleftrightarrow\ 
\lim_{k\to\infty}\int_{\partial\Omega_k}\!\!\sigma\nabla u.n\,\psi\,d\Lambda=
\langle \sigma g\,, \,\tr_{\partial\Omega}\psi\rangle,\quad \psi\in W^{1,2}_\RR(\Omega),
\end{equation}
where $n$ denotes the unit normal on $\partial\Omega_k$, irrespective of $k$.
Elaborating on this,  we let $\mathfrak{n}$ indicate the complex number
$n_x+in_y$ where $(n_x,n_y)^t=n$, and we observe upon
making use of \eqref{system} that
\[\sigma \, \partial u \, \mathfrak{n}=\sigma\,\frac{\partial_{x_1}u-i\partial_{x_2}u}{2}\,\mathfrak{n}=
\left(\sigma \, \nabla u+i \, \nabla v \right).\frac{n}{2} \, ,\]
where ``$.$'' indicates the Euclidean scalar product.
In view of \eqref{seqGreen} and its analog for $v$
(remember $v$ satisfies \eqref{forcond} with $\sigma$ replaced by $1/\sigma$),
we  obtain:
\[
u\in W^{1,2}(\Omega) \text{ satisfies \eqref{forcond} with }  \partial_nu= g   \text{ and } \partial_\tau u=-h
\]
\begin{equation}
\label{seqGreenc}
\Longleftrightarrow\ 
2\lim_{k\to\infty}\int_{\partial\Omega_k}\!\!\!\!\!\sigma \, \partial u \, \mathfrak{n} \,
\psi\,d\Lambda=
\langle \sigma \, (g+ih)\,, \,\tr_{\partial\Omega}\psi\rangle \, , \ \psi\in W^{1,2}(\Omega)  \, ,
\end{equation}
where we complexified the space of test functions 
({\it i.e.} from $\psi\in W^{1,2}_\RR(\Omega)$ to $\psi\in W^{1,2}(\Omega)$)
upon extending the pairing $\langle\,,\,\rangle$ in a complex-linear manner. 

\subsubsection{Factorization of the complex derivative}
\label{facder}
The lemma below substantially reduces the study of solutions to 
\eqref{forcond}, when $n=2$ and \eqref{isot}-\eqref{ellip} hold,
to that of harmonic functions.
\begin{lemma}
\label{BNu}
Let $u\in W^{1,2}_\RR(\Omega)$ satisfy \eqref{forcond} on a bounded 
simply connected Lipschitz domain $\Omega\subset\RR^2$, with $\sigma$ subject
to \eqref{isot} and \eqref{ellip}. Then, there exists
a holomorphic function $F\in L^2(\Omega)$, a number $r_1\in(2,r]$,
{ a function $\Upsilon\in W^{1,r_1}(\Omega)$
with real-valued $\tr_{\partial\Omega}\Upsilon$ 
whose} norm is bounded solely in terms of $\Omega$, 
$\|\sigma\|_{W^{1,r}(\Omega)}$, $r$,  
and ellipticity constants in \eqref{ellip}, such that
$\partial u=e^\Upsilon F$. Moreover $F=\partial H$ where 
$H\in W^{1,2}_\RR(\Omega)$ is harmonic in $\Omega$ and satisfies
on $\partial\Omega$:
\begin{equation}
\label{deriveesHr}
\partial_n H=e^{-\Upsilon}\partial_n u,
\qquad
\partial_\tau H=e^{-\Upsilon}
\partial_\tau u.
\end{equation}
\end{lemma}
\begin{proof}
Let $f=u+iv$ where $v$ is the $\sigma$-harmonic conjugate to $u$.
Since $f\in W^{2,r}_{loc}$ is a fortiori locally bounded and
 satisfies \eqref{CB} with 
$\nu\in W^{1,r}_\RR(\Omega)$ and $\|\nu\|_{L^\infty(\Omega)}<1$,  
a short computation as  in the proof of \cite[Cor. 3.3]{BFL} or
\cite[Lem. 5]{BLRR}
shows that 
$w:=(1-\nu^2)^{1/2}\partial f$ satisfies 
$\bar\partial w=(\partial\nu/(1-\nu^2))\bar w$. As
$\partial\nu/(1-\nu^2)\in L^r(\Omega)$ and $r>2$, the Bers similarity
principle for pseudo-holomorphic functions (see {\it e.g.}
\cite[Prop. 3.2]{BFL})
entails that
there exist $s\in W^{1,r}(\Omega)$, whose norm\footnote{Reference \cite{BFL} deals with Dini-smooth
$\Omega$ but this assumption is not used in the proof of 
equations (18), (19) {\it loc. cit.} The similarity principle
is called \emph{representation of pseudo-analytic functions of the first kind}
in \cite[Ch. III, Sec. 4]{vekua}, { and later appeared in many works}.} is bounded in terms of 
$r$ and $\|\partial\nu/(1-\nu^2)\|_{L^r(\Omega)}$ only,
and also a holomorphic function $F_1$ on $\Omega$ such that
$w=e^s F_1$. Hence $\partial f=e^{s_1}F_1$, where 
$s_1=s-\log(1-\nu^2)^{1/2}$ belongs to $W^{1,r}(\Omega)$ by \eqref{isot},
\eqref{ellip} { and {\eqref{estfonc}}}.
Now, it is straightforward to check
using \eqref{system} that $\partial f=(1+\sigma)\partial u$.
Therefore, if we set $\Upsilon_1=s_1-\log(1+\sigma)$ and appeal { again} to
\eqref{isot},
\eqref{ellip} and \eqref{estfonc}, we get that
\begin{equation}
\label{fac1}
\partial u=e^{\Upsilon_1}F_1, \qquad \Upsilon_1\in W^{1,r}(\Omega),
\quad F_1\textrm{\ holomorphic in } \Omega \, ,
\end{equation}
where we notice that $\|\Upsilon_1\|_{W^{1,r}(\Omega)}$ is
bounded in terms of $\Omega$, $r$, the constants in \eqref{ellip}
and $\|\sigma\|_{W^{1,r}(\Omega)}$.
Factorization \eqref{fac1} is not yet what we need,
for $\tr_{\partial\Omega}\Upsilon_1$ may not be real-valued.
To remedy this, we will trade $r$ for a possibly smaller 
exponent { {$r_1>2$}}.
Specifically, it follows from \cite[Thm 5.1]{JKinhom}\footnote{The result 
is stated there for $n\geq3$ only, which may be confusing,
but the proof is valid for  $n=2$ as well. In fact, all we { need 
is}  \cite[Thm 5.15, (a),(b)]{JKinhom} for Besov spaces,
along with interpolation arguments on  top 
of \cite[p. 200]{JKinhom}. 
Since that part of the proof of \cite[Thm 5.15]{JKinhom} depends 
only on \cite{DahlbergNT}, complex interpolation and multiplier theory
for singular integral operators, the restriction $n\geq3$ is easily seen to be 
superfluous.} that the Dirichlet problem for harmonic functions
with boundary values in $W^{\theta,p}(\partial\Omega)$ is solvable in
$W^{\theta+1/p,p}(\Omega)$, as soon as $0<\theta<1$ and
$p\in[2,2+\varepsilon)$ where $\varepsilon>0$ depends on the Lipschitz 
constant of $\partial \Omega$. 
{ Since 
$\tr_{\partial\Omega}\Upsilon_1\in W^{1-1/r,r}(\partial\Omega)$} and
the latter space increases as $r$ decreases, there is
$r_1\in(2,r]$ (depending on $\Omega$ and $r$) and a harmonic function
$h\in W^{1,r_1}(\Omega)$ such that $\tr_{\partial\Omega}h=\tr_{\partial\Omega}
\textrm{Im}\Upsilon_1$. Let $g$ be a harmonic conjugate to $h$, normalized 
so that $ g_E=0$ for some $E\subset\Omega$ with $m_2(E)>0$.
Since $|\nabla h|=|\nabla g|$ pointwise by the Cauchy-Riemann equations, it
follows from \eqref{estfonc} that $g$ lies in $W^{1,r_1}(\Omega)$, and
so do the holomorphic functions $G:=-g+ih$ and $e^{G}$ since $\exp$ is entire.
Setting 
\begin{equation}
\label{nouveauUpsilonF}
\Upsilon=\Upsilon_1-G\qquad \textrm{ and}\qquad  F=F_1e^{G},
\end{equation}
 we have that 
$\Upsilon\in W^{1,r_1}(\Omega)$ is real-valued on $\partial\Omega$ and that
$F$ is holomorphic, while $\partial u=e^{\Upsilon}F$, as desired.

Because $u\in W^{1,2}_\RR(\Omega)$ by assumption
and $\Upsilon\in L^\infty(\Omega)$  by the 
Sobolev embedding theorem, we get that $F\in L^2(\Omega)$. Being holomorphic, 
$F$ can be written as $\partial H$ for some real-valued harmonic $H$, and
necessarily $H\in W^{1,2}_\RR(\Omega)$ for its complex derivatives $\partial H$ and
$\bar\partial H=\overline{\partial H}$ are 
in $L^2(\Omega)$. Being harmonic, $H$ satisfies \eqref{forcond} with 
$\sigma$ replaced by $1$, so we get the following analog to 
\eqref{seqGreenc}:
\[
\partial_n \, H= \gamma \mbox{ and } \partial_\tau H=-\mu 
\]
\begin{equation}
\label{seqGreencH}
\Longleftrightarrow\ 
2\lim_{k\to\infty}\int_{\partial\Omega_k}\!\!\!\!\!\partial H \, \mathfrak{n} \, 
\phi\,d\Lambda=
\langle \gamma+i\mu \,, \,\tr_{\partial\Omega}\phi\rangle,\ \phi\in W^{1,2}(\Omega).
\end{equation}
Substituting $\partial u=e^\Upsilon \partial H$ in \eqref{seqGreenc}
and reckoning that $\psi\mapsto \phi=\sigma e^\Upsilon \psi$ is an isomorphism
of $W^{1,2}(\Omega)$ because $\sigma e^{\Upsilon}\in W^{1,r_1}(\Omega)$, 
we see from \eqref{seqGreencH} that
\begin{equation}
\label{deriveesH}
\partial_n H=\text{\rm Re}\Bigl(e^{-\Upsilon}(\partial_n u-i\partial_\tau u)\Bigr),
\quad
\partial_\tau H=-\text{\rm Im}\Bigl(e^{-\Upsilon}
(\partial_n u-i\partial_\tau u)\Bigr).
\end{equation}
{ Taking into account} in \eqref{deriveesH}
that $e^\Upsilon$ is real-valued on $\partial\Omega$
yields \eqref{deriveesHr}.
\end{proof}

\subsubsection{Proof of Theorem \ref{normeq3/2}}
\label{pf3/2}
\begin{proof} 
Write $\partial u=e^{\Upsilon}F$ as in Lemma \ref{BNu},
and let $H\in W^{1,2}_\RR(\Omega)$ be a
harmonic function such that $F=\partial H$.
Since 
$e^{-\Upsilon}$ is bounded, { being continuous} on $\overline{\Omega}$
by the Sobolev embedding theorem, we deduce from \eqref{deriveesHr} 
{ that  $\partial_n H$ lies} in $ L^2(\partial\Omega)$.
Set $G$ to be a harmonic conjugate to $H$. By the discussion after
\eqref{def2nd} ({ with $\sigma\equiv1$} throughout),
we get that $G$ is, up to an additive  constant, the unique  
harmonic function 
in $W^{1,2}_\RR(\Omega)$ such that
$\partial_\tau G=\partial_n H$. 
From Proposition \ref{Neucor}, we now see that 
$\partial G\in\mathcal{S}^2(\Omega)$ with
$\textrm{Re}(\partial G \,\mathfrak{t})=e^{-\Upsilon}\partial_n u/2$ on $\partial\Omega$,
where $\mathfrak{t}$ is the tangent vector field to 
$\partial\Omega$ written in complex form.
By the Cauchy-Riemann equations $\partial H=i\partial G$,
so that in turn $F\in\mathcal{S}^2(\Omega)$ and
$\textrm{Im}(F \,\mathfrak{t})=e^{-\Upsilon}\partial_n u/2$ on $\partial\Omega$.
Also $H\in W_\RR^{3/2,2}(\Omega)$ for  $\partial H=F$ lies in 
$W^{1/2,2}(\Omega)$ by Theorem \ref{equivS2}  $(ii)$.
Let $\mathfrak{n}=\mathfrak{t}/i$ be the normal vector field on $\partial\Omega$, written in complex form.
By definition of complex derivatives 
({\it cf.} \eqref{defderC}), the nontangential convergence of $\nabla u$ to 
{ {$\partial_n u\,n+\partial_\tau u\,\tau$  is equivalent}} to
the nontangential convergence of $\partial u=e^{\Upsilon}F$ to
$\partial_n u\,\overline{\mathfrak{n}}/2+\partial_\tau u\,\overline{ \mathfrak{t}}/2$.
Considering the existence of nontangential limits { a.e.} for
Smirnov functions and the continuity of $\Upsilon$,
this is in turn equivalent to 
$2F=e^{-\Upsilon}(\partial_n u\,\overline{\mathfrak{n}}+\partial_\tau u\,\overline{ \mathfrak{t}})$, that is 
$2F\mathfrak{t}=e^{-\Upsilon}(i\partial_n u+\partial_\tau u)$ on $\partial\Omega$.
Taking real and imaginary 
parts, we are thus left { to verify  two} real equations:
\[2\textrm{Re}(F\mathfrak{t})=e^{-\Upsilon}\partial_\tau u\qquad
\textrm{ and}\qquad2\textrm{Im}(F\mathfrak{t})=e^{-\Upsilon}\partial_n u.
\] The second of these has already been checked.
By \eqref{deriveesHr}, the first reduces to
$2\textrm{Re}(\partial H\mathfrak{t})=\partial_\tau H$ which holds good by
\eqref{exptd}. That 
$u\in W_\RR^{3/2,2}(\Omega)$ follows from the relation
$\partial u= e^{\Upsilon} F$, the membership $F\in W^{1/2,2}(\Omega)$,
and the
fact that $e^\Upsilon$ is a multiplier of
$W^{1/2,2}(\Omega)$ for it lies in
$W^{1,r_1}(\Omega)$ { with $r_1>2$. Because} 
$\|\Upsilon\|_{W^{1,r_1}(\Omega)}$ depends only on 
$\Omega$, $\|\sigma\|_{W^{1,r}(\Omega)}$, $r$,  
and ellipticity constant $c$ in \eqref{ellip}, as asserted by
Lemma \ref{BNu}, so does the norm of this multiplier. 
Combining this with Theorem \ref{equivS2}  $(ii)$, we obtain:
\begin{equation}
\label{gradu1/2}
\|\nabla u\|_{W^{1/2,2}(\Omega)}\leq C\|\partial H\|_{W^{1/2,2}(\Omega)}=
C\|\partial G\|_{W^{1/2,2}(\Omega)}
 \leq
C'\|\partial G\|_{\mathcal{S}^2(\Omega)},
\end{equation}
where $C'$ depends only on the above-mentioned parameters. Besides,
we get from 
Proposition  \ref{Neucor} 
(applied with  $U=G$)
and \eqref{deriveesHr} that 
\[\|\partial G\|_{\mathcal{S}^2(\Omega)}\leq 
C''\|\partial_\tau G\|_{L^2(\partial\Omega)}=
C''\|\partial_n H\|_{L^2(\partial\Omega)}
\leq
C''e^{\|\Upsilon\|_{L^\infty(\partial\Omega)}} \|\partial_n u\|_{L^2(\partial\Omega)}
\]
where $C''$ depends only on $\Omega$. The latter estimate and
\eqref{gradu1/2} together yield
\begin{equation}
\label{estgradu1/2}
\|\nabla u\|_{W^{1/2,2}(\Omega)}\leq C_0 \|\partial_n u\|_{L^2(\partial\Omega)},
\end{equation}
where $C_0$ depends on the same parameters as $C'$.

Now, if we normalize $u$ so that $u(z_0)=0$ for some 
$z_0\in\Omega$  and
let $\rho_0>0$ be such that $\DD(z_0,\rho_0)\subset\Omega$, 
we deduce  from \eqref{forcond}, \eqref{ellip}, the Green formula
(which is valid since $u\in W^{2,r}_{\RR,loc}(\Omega)$) and
H\"older's inequality 
that, for $\rho\in(0,\rho_0)$,
\begin{align}
\nonumber
&\left|\frac{d}{d\rho}\left(\frac{1}{\rho}\int_{\TT(z_0,\rho)}\!\!\!\!\!\!u\,dm
\right)\right|=
\left|\frac{1}{\rho}\int_{\TT(z_0,\rho)}\!\!\!\!\!\!\partial_n u\,dm\right|
\leq\frac{1}{c\rho}\left|\int_{\TT(z_0,\rho)}\!\!\!\!\!\!\sigma\partial_n u\,dm\right|\\
\label{estimoy}
&=\frac{1}{c\rho}\left|\int_{\DD(z_0,\rho)}\!\!\nabla u.\nabla\sigma\,dm_2
\right|
\leq\frac{\pi^{1/4}}{c\rho^{1/2}}\|\nabla u\|_{L^4(\Omega)}\|\nabla\sigma\|_{L^2(\Omega)}{ .}
\end{align}
Since 
$\lim_{\rho\to0}\int_{\TT(z_0,\rho0)}udm/\rho=u(z_0)=0$, we infer from 
\eqref{estgradu1/2}, \eqref{estimoy} and the Sobolev embedding theorem
that, for $\rho\in(0,\rho_0)$,
\begin{equation}
\label{moycercu}
\left|\frac{1}{\rho}\int_{\TT(z_0,\rho)}\!\!\!\!\!\!u\,dm
\right|\leq C_1 \rho^{1/2}\|\partial_n u\|_{L^2(\partial\Omega)}
\end{equation}
where $C_1$ depends on 
$\Omega$, $\|\sigma\|_{W^{1,r}(\Omega)}$, $r$,  
and $c$ 
in \eqref{ellip}. Integrating \eqref{moycercu}
yields
\begin{equation}
\label{moydiscu}
\left|\frac{1}{\pi \rho_0^2}\int_{\DD(z_0,\rho_0)}u\,dm_2\right|\leq
C_2\rho_0^{1/2}\|\partial_n u\|_{L^2(\partial\Omega)}
\end{equation}
where $C_2$ depends on the same parameters as $C_1$.
Then,  \eqref{domnu3/2} follows from \eqref{estfonc}, \eqref{estgradu1/2} and \eqref{moydiscu}.

Finally, remember from \eqref{nouveauUpsilonF} the factorization 
$\partial u=e^{\Upsilon_1}F_1$ where $\Upsilon_1\in W^{1,r}(\Omega)$ and 
$F_1=F e^{-G}$ with $G\in W^{1,r_1}(\Omega)$, $G$ holomorphic.
{ As} $G\in L^\infty(\Omega)$ by the Sobolev embedding theorem,
we see that $F_1\in\mathcal{S}^2(\Omega)$ because $F$ does, so we may set
$\Psi=\Upsilon_1$ and $\Phi=F_1$, thereby completing the proof of $(i)$.

In view of \eqref{deriveesHr}, the factorization $\partial u=e^{\Upsilon}F$,
 and the boundedness of
$e^\Upsilon$, the proof of \eqref{equivmax3/2} reduces to the case where
$u$ is harmonic ({\it i.e.} $\sigma\equiv1$), and then it follows from 
Theorem \ref{equivS2} and Proposition \ref{Neucor} applied to $u$ and its
conjugate function. This shows $(ii)$.

As to $(iii)$, we already mentioned that $u\in W^{2,r}_{\RR,loc}(\Omega)$ 
(this is now obvious anyway since  
$\partial u=e^{\Psi}\Phi$) and we need to prove 
\eqref{equivpond3/2} which is equivalent to 
\begin{equation}
\label{reduc1/2}
\sum_{j=1,2}\|\textrm{d}(.,\partial\Omega)^{1/2}\,\partial_{x_j}(e^{\Upsilon}F)\|^2_{L^2(\Omega)}
+\|e^{\Upsilon}F\|_{L^2(\Omega)}^2
\sim\|e^\Upsilon F\|_{W^{1/2,2}(\Omega)}^2 \, .
\end{equation}
We already know from \eqref{inegfracdb} that the right hand side of \eqref{reduc1/2} is less than a constant (depending only on $\Omega$) 
times the left hand side.
To prove the reverse inequality, let $\varphi$ conformally map 
$\DD$ onto $\Omega$, so that
$f:=(F\circ\varphi)(\varphi')^{1/2}\in H^2$, and  recall that 
$|f(z)|\leq c\|f\|_{H^2}/(1-|z|)^{1/2}$ for $z \in \DD$ and some absolute constant $c$, 
by a classical inequality
of Hardy and Littlewood \cite[Thm 5.9]{duren}. Since 
$\textrm{d}(\varphi(z),\partial\Omega)\leq (1-|z|^2)|\varphi'(z)|$
by standard properties of conformal maps 
\cite[Ch.1, Cor. 1.4]{Pommerenke}, we get that 
$\|\textrm {d}(.,\partial\Omega)^{1/2} F\|_{L^\infty(\Omega)}\leq
\sqrt{2}c\|F\|_{\mathcal{S}^2(\Omega)}$.
Now,  by the Leibniz rule and the triangle inequality,
the first summand in the left hand side of \eqref{reduc1/2} is 
{ bounded above} by
\[
\sum_{j=1,2}\|\textrm{d}(.,\partial\Omega)^{1/2}(\partial_{x_j}\Upsilon) e^{\Upsilon}F\|^2_{L^2(\Omega)}
+
\sum_{j=1,2}\|\textrm{d}(.,\partial\Omega)^{1/2}e^\Upsilon
(\partial_{x_j}F)\|^2_{L^2(\Omega)}
\]
which is less than
\[\|e^\Upsilon\|_{L^\infty(\Omega)}^2 \, \left( 2\,{c}^2  \, 
\|\nabla\Upsilon\|_{L^2(\Omega)}^2 \, \|F\|_{\mathcal{S}^2(\Omega)}^2+
\sum_{j=1,2}\|\textrm{d}(.,\partial\Omega)^{1/2}
(\partial_{x_j}F)\|^2_{L^2(\Omega)}
\right).
\]
By \eqref{W1/2S2} and Theorem \ref{equivS2} $(ii)$, this 
quantity is majorized by $c'\|F\|_{W^{1/2,2}(\Omega)}^2$ where 
$c'=c'(\Omega, r_1,\|\Upsilon\|_{W^{1,r_1}(\Omega)})$, and since
$\|e^{\Upsilon}F\|_{W^{1/2,2}(\Omega)}\sim \|F\|_{W^{1/2,2}(\Omega)}$ 
we are done with the proof.
\end{proof}

\subsubsection{A generalized Rolle's theorem}
\label{sec:hardy}
We recall below the 1-dimensional version
of a Lusin-type theorem for Sobolev functions,
to be found in \cite[Thm 3.10.5]{ziemer}. 
More precisely, we { state the} case
$n=1$, $\ell=k=1$ and $p=2$ of the result just quoted.
The latter is in terms of Bessel capacities that we did not introduce, but we 
use here that the Bessel capacity $B_{0,2}$ is just Lebesgue measure,
see \cite[Def. 2.6.2]{ziemer}.
\begin{lemma}{\cite[Thm 3.10.5]{ziemer}}
\label{LusinSob}
Let $v \in W^{1,2}(\RR)$ and $\eps>0$. There exists an open set $\mathcal{U} 
\subset \RR$ and a function $w \in C^1(\RR)$ such that
$m_1(\mathcal{U} )<\eps$ and
\[w(t)=v(t),\quad w'(t)=v'(t),\qquad \forall t\in\RR \setminus \mathcal{U} .\]
\label{lemma1}
\end{lemma}
We use Lemma \ref{LusinSob} to prove the following generalization of Rolle's theorem.
\begin{proposition}
Let $\Omega\subset\RR^2$ be a bounded Lipschitz domain and 
$v\in W^{1,2}(\partial\Omega)$. Assume that $v=0$ at 
$\Lambda$-a.e. point of a set $B\subset\partial\Omega$ with $\Lambda(B)>0$.
Then, there is $B'\subset B$, with $\Lambda(B')>0$, such that 
$\partial_\tau v=0$ at $\Lambda$-a.e. point of $B'$.
\label{real}
\end{proposition}
\begin{proof}
As $\Omega$ is bounded and Lipschitz, 
$\partial\Omega$
can be covered with open parallelepiped  $Q_1,\cdots,Q_N$ of the form
$Q_j=\mathcal{R}_j(Q_{a,b})$, where $\mathcal{R}_j$ is an
affine isometry of $\RR^2$ and $Q_{a,b}=(-a,a)\times(-b,b)$, with $a,b>0$,
in such a way that 
\[
\mathcal{R}_j^{-1}(\Omega) \cap Q_{a,b}=\{x\in Q_{a,b}:\,x_2> \psi_j (x_1)\},
\]
where $\psi_j$ is a Lipschitz function.
Denote by $P_1$ the projection 
onto the first component in $\RR^2$, and for $E\subset \partial\Omega\cap Q_j$
set $E_j=P_1(\mathcal{R}_j^{-1}(E))\subset(-a,a)$ so that
$\Lambda(E)=\int_{E_j} |(1,\psi_j')^t|dm_1$.
Since
Lipschitz changes of variables preserve Sobolev classes
\cite[Thm 2.2.2]{ziemer}, it holds 
that $v\in W^{1,2}(\partial\Omega)$ 
if and only if
$v(\mathcal{R}_j(x_1,\psi_j(x_1)))$ belongs to 
$W^{1,2}((-a,a))$ for all $j$. Thus, as $\Lambda(B\cap Q_j)>0$ for at 
least one $j$,
it is enough to prove the analog of the proposition 
on the real interval $(-a,a)$ instead of $\partial\Omega$. 
By the extension theorem we may assume that
$v$ is defined over the whole
real line. Then, taking $\varepsilon$ small enough in Lemma \ref{LusinSob},
we conclude that it is enough to prove Proposition \ref{real}
when $v$ has continuous derivative. Assume it is the case and
let $A$ be the set of accumulation points of $B$.
Note that $m_1(A)=m_1(B)>0$ because
$B\setminus A$ is countable \cite{jech}. Moreover,
to each $t\in A$,
there exists a non-stationary sequence $(t_n)\subset B$, $n \in \mathbb{N}$, 
such that $t_n \rightarrow t$.  
Without loss of generality, we may assume that $(t_n)$ is monotone,
say  $t_n <t_{n+1}$ (the case $t_n>t_{n+1}$ is similar) for all 
$n \in \mathbb{N}$. Since $v(t_n)=v(t_{n+1})=0$,
there is $s_n\in[t_n,t_{n+1}]$ such that $v'(s_n)=0$
by Rolle's theorem. Since $s_n\rightarrow t$ and $v'$ in continuous, we get that $v'(t)=0$, as desired.
\end{proof}

\subsection{Proof of Theorem \ref{cauchy1}}
\label{subs:proof1}

\begin{proof} 
From Theorem \ref{normeq3/2}  $(ii)$,
we get that $u_{|\partial\Omega}$ lies in $W^{1,2}(\partial\Omega)$,
hence Proposition \ref{real} implies that both
$\partial_n u$ and $\partial_\tau u$ vanish on some $E\subset\gamma$ with
$\Lambda(E)>0$. By Theorem  \ref{normeq3/2}  $(i)$, we now see that
$(e^{\Psi} \Phi)_{|\partial\Omega}$ vanishes a.e. on $E$, and since 
$e^{-\Psi}$ is bounded, by the Sobolev embedding theorem, we must have 
that $F_{|\partial\Omega}=0$ a.e. on $E$. As $F$ belongs to 
the Smirnov class $\mathcal{S}^2(\Omega)$, we { deduce 
that $F\equiv0$}, hence $\partial u\equiv0$ and thus  $\nabla u\equiv0$ since $u$ is real.
Therefore $u$ is a constant, and in fact $u\equiv0$.
\end{proof}

\section{The anisotropic case}
\label{anisotropic}
We consider in this section
a conductivity equation of the form 
\eqref{forcond}  
where $\sigma$ is valued in the set of real symmetric matrices and
the ellipticity condition \eqref{ellip} is replaced by
\begin{equation}
\label{ellipa}
c \, I_n \leq \sigma \leq c^{-1} \, I_n 
\quad\text{for some constant}\ c\in(0,+\infty),
\end{equation}
with $I_n$ the identity matrix of order $n$. 
Isotropic equations 
correspond to the case where the image  of $\sigma$
consists of scalar matrices; otherwise the conduction
is said to be anisotropic.
Existence and uniqueness of solutions to the Neumann problem and the
forward Robin problem proceed as before provided
that the normal derivative gets replaced by $n.\sigma\nabla u$.
Questions about uniqueness of the
Robin coefficient for the inverse problem may be raised 
as in Section \ref{sec:cond}, namely: \emph{given 
$0\not\equiv g\in L^2(\Gamma_0)$ and $u$ the solution to the forward Robin 
problem with Robin coefficient $\lambda\in L_+^\infty(\Gamma)$, subject to
the boundary condition $n.\sigma\nabla u=g$, does the 
knowledge of $u_{|\Gamma_0}$ determine $\lambda$ uniquely?}

Of course when $n\geq3$, uniqueness cannot  prevail in general
as we saw it may not even hold for the 
ordinary Laplacian. But if $n=2$ uniqueness does hold: this 
follows from Theorem  \ref{cauchy1} and the fact that an anisotropic 
equation 
on a bounded Lipschitz domain with $W^{1,r}$ coefficients, $r>2$,
is the diffeomorphic image of an isotropic one
(on another Lipschitz domain).
More precisely, if $\Theta$ is a diffeomorphism of $\RR^2$ of class $C^1$
and if we set $\Omega_1=\Theta(\Omega)$, a computation shows 
(see {\it e.g.} \cite{APL}) that
$u$ solves for \eqref{forcond}  in $W^{1,2}(\Omega)$ if and only if
$v=u\circ\Theta^{-1}$ solves for 
$\nabla \cdot \left( \tilde{\sigma} \, \nabla v \right) =0$
in $W^{1,2}(\Omega_1)$, where
\begin{equation}
\label{tildesigma}
\tilde{\sigma}(\Theta(z))=\frac{1}{|D\Theta(z)|}
{D\Theta(z)\sigma(z)D\Theta^t(z)}
\end{equation}
and $|D\Theta|$ indicates the determinant of the Jacobian matrix $D\Theta$.
Moreover,  using a subscript 1 for the unit tangent and normal vectors to
$\Omega_1$, it holds by construction that
$\partial_{\tau_1} v\circ\Theta=\partial_\tau u/|D\Theta\tau|$, and 
from the weak 
formulation of the Neumann problem we get that 
$(n_1.\tilde{\sigma}\nabla v)\circ\Theta=n.\sigma\nabla u/|D\Theta\tau|$.

Now, since $\sigma=(\sigma_{ij})$ has entries in $W^{1,r}(\Omega)$
and satisfies \eqref{ellipa},  we can 
extend it into a symmetric matrix-valued function with entries in 
$W^{1,r}_{loc}(\RR^2)$ meeting \eqref{ellipa}  
and equal to $I_2$
outside of a compact set; this only requires the extension theorem, 
continuity of
$W^{1,r}$-functions when $r>2$, and a smooth partition of unity.
Denoting this extension by $\sigma$ again,
define the complex function 
$\mu_1=(-\sigma_{11}+\sigma_{2,2}-2i\sigma_{12})/(\sigma_{11}+\sigma_{22}+2\sqrt{|\sigma|})$. As $\mu_1$ is compactly supported and $|\mu_1|<C<1$, the 
solution $\Theta$ to the
Beltrami equation $\bar\partial \Theta=\mu_1\partial\Theta$  which is
$z+O(1/z)$ at infinity is a homeomorphism of $\CC$ of class $W^{2,r}_{loc}$
(a fortiori it is $C^1$-smooth) and 
 $\tilde\sigma$ given by \eqref{tildesigma} satisfies
$\tilde\sigma=|\sigma\circ\Theta^{-1}|^{1/2}$ \cite{SunUhl}, see also
\cite{Sylvester}
where this technique was initiated for smoother coefficients and the nice 
exposition in \cite{APL} which deals with bounded coefficients 
(but less smooth $\Theta$).  Because $\Omega_1=\Theta(\Omega)$ 
is Lipschitz and
the scalar-valued function
$\tilde\sigma$ satisfies \eqref{isot} and \eqref{ellip}
(with $n=2$), we can apply Theorem \ref{cauchy1} to $v$ on $\Omega_1$.
Thus, we deduce from the relations between $u$ and $v$ that if
$u\not\equiv0$ then $u$ and 
$n.\sigma\nabla u$
cannot vanish together on a subset of positive measure of $\partial\Omega$.
The proof of Theorem \ref{main1} can now be repeated to give us:
\begin{corollary}
Theorem \ref{main1} still holds in the anisotropic case 
when $\sigma$ is real symmetric $2\times2$-valued 
with  entries in $W^{1,r}(\Omega)$ and 
meets \eqref{ellipa},
$r>2$, provided the normal derivative $\partial_n u$ in \eqref{forward1}
gets replaced by $n.\sigma\nabla u$. 
\end{corollary}

\section{{ Concluding remarks}} 
\label{conclusion}
{ 
In the notation of \eqref{forward10}, stable determination
with respect to $u_{|\Gamma_0}$,
of a smooth Robin coefficient $\la$ 
has been studied in \cite{alessandrini_delpiero_rondi,
chaabane_fellah_jaoua_leblond,chaabane_jaoua,
cheng_choulli_lin,
sincich} for the Laplace equation. When $n=2$, the
factorization of $\partial u$ given in  Theorem \ref{normeq3/2} 
may
help dealing with this issue for more general conductivities and
less smooth $\la$.

We further mention that 
stability of the Cauchy problem in dimension 2,
for general anisotropic conductivity equations
with bounded conductivity,
has been extensively studied in \cite{alessandrini} using
tools from complex analysis. 

In this connection, we point out that the factorization technique of  
Lemma \ref{BNu} enjoys some generalization
to the anisotropic case which rests on the method of isothermal coordinates 
that we recalled in
Section \ref{anisotropic}. This suggests a research path worth exploring 
when dealing with stability  for Sobolev-smooth conductivities.

It is also natural to ask whether  results from the present paper remain 
valid when $\sigma$ is merely bounded. At present, 
the derivation of a factorization for $\partial u$ requires some smoothness
and it is not even clear if it holds for $\sigma\in L^\infty\cap W^{1,r}$ 
when $r<2$. The case where $\sigma\in W^{1,2}(\Omega)$ deserves special 
mentioning: although the equation may no longer be strictly elliptic and
solutions need not even be locally bounded, it is possible to make sense out 
of the Dirichlet problem for $L^p$ data and the factorization 
$\partial u=e^\Psi\Phi$ still holds in slightly modified form where 
$\Psi$ is pure imaginary on $\partial\Omega$ \cite{BBC}. Therefore
we expect some of our theorems to remain valid, at least if $\Omega$ is smooth.

Yet another generalization concerns with complex-valued $\sigma$, which 
arise in  impedance tomography \cite{CIN}. In this case
\eqref{system} becomes a system of complex equations, and the factorization 
of $\partial u$ has apparently not been investigated.

Since the negative result of \cite{bourgain_wolff},
weaker unique continuation issues have been raised in dimension 3 and higher.
One of them is: does a harmonic
function in $\Omega$,  the trace of which vanishes on a non-empty  open 
subset {{${\mathcal O}$}} of  $\partial\Omega$ and whose normal derivative vanishes on a subset 
of positive measure { in ${\mathcal O}$}, have to vanish identically?
This question is still open in general, and we refer the reader to
\cite{adolfsson_escauriaza,kenig_wang} for advances on the subject. 
{ In} the setting of Robin inverse problems, the issue raised
in Remark \ref{rmk2} as to whether $\partial_nu/u$ can remain non-negative 
and bounded in a neighborhood of a set of positive measure where $u$,
$\partial_nu$ both vanish,  seems to be 
more relevant and deserves further study.
}

Finally, we did not touch upon multiply connected domains $\Omega$, where similar uniqueness properties can be proved.

\section*{Acknowledgments} 
The authors thank the reviewer for his constructive remarks.

\bibliographystyle{plain} 
\bibliography{biblioBBL1}

\appendix

\end{document}